\documentclass[12pt,reqno]{amsart}
\usepackage{moreverb}
\usepackage{color}
\usepackage{amsmath,amssymb,amsfonts}
\usepackage{graphicx}
\usepackage{subfigure}
\usepackage{fullpage}

\definecolor{darkblue}{rgb}{0,0,.8}
\definecolor{darkred}{rgb}{0.8,0,0}

\usepackage{hyperref}
\hypersetup{
  pdftex=true,
  colorlinks=true,
  allcolors = darkred,
  linkcolor=darkblue, 
  menucolor=darkblue,
  urlcolor = darkblue,
  frenchlinks=false,
  pdfborder={0 0 0},
  naturalnames=false,
  hypertexnames=false,
  breaklinks
}

\usepackage[capitalize]{cleveref}
\crefname{section}{section}{sections}

\Crefname{figure}{Figure}{Figures}

\crefformat{equation}{\textup{#2(#1)#3}}
\crefrangeformat{equation}{\textup{#3(#1)#4--#5(#2)#6}}
\crefmultiformat{equation}{\textup{#2(#1)#3}}{--\textup{#2(#1)#3}}
{, \textup{#2(#1)#3}}{, and \textup{#2(#1)#3}}
\crefrangemultiformat{equation}{\textup{#3(#1)#4--#5(#2)#6}}%
{ and \textup{#3(#1)#4--#5(#2)#6}}{, \textup{#3(#1)#4--#5(#2)#6}}{, and \textup{#3(#1)#4--#5(#2)#6}}

\Crefformat{equation}{#2Equation~\textup{(#1)}#3}
\Crefrangeformat{equation}{Equations~\textup{#3(#1)#4--#5(#2)#6}}
\Crefmultiformat{equation}{Equations~\textup{#2(#1)#3}}{ and \textup{#2(#1)#3}}
{, \textup{#2(#1)#3}}{, and \textup{#2(#1)#3}}
\Crefrangemultiformat{equation}{Equations~\textup{#3(#1)#4--#5(#2)#6}}%
{ and \textup{#3(#1)#4--#5(#2)#6}}{, \textup{#3(#1)#4--#5(#2)#6}}{, and \textup{#3(#1)#4--#5(#2)#6}}

\crefdefaultlabelformat{#2\textup{#1}#3}

\newtheorem{theorem}{Theorem}
\newtheorem{lemma}[theorem]{Lemma}

\newtheorem{problem}[theorem]{Problem}
\newtheorem{corollary}[theorem]{Corollary}

\newtheorem{remark}[theorem]{Remark}

\def\subsection#1
{
 \bigskip

 \refstepcounter{subsection}
 {\bf\arabic{section}.\arabic{subsection}.~#1.~}
}

\def\subsubsection#1
{
 \bigskip

 \refstepcounter{subsubsection}
 {\bf\arabic{section}.\arabic{subsection}.\arabic{subsubsection}.~#1.~}
}

\usepackage{pgfplots}
\usepackage{pgfplotstable}
\usetikzlibrary{calc}
\usetikzlibrary{plotmarks}
\usetikzlibrary{arrows.meta}
\usepgfplotslibrary{patchplots}
\pgfplotsset{plot coordinates/math parser=false}
\newlength\figureheight
\newlength\figurewidth 

\newcommand{\logLogSlopeTriangle}[6]
{

    \pgfplotsextra
    {
        \pgfkeysgetvalue{/pgfplots/xmin}{\xmin}
        \pgfkeysgetvalue{/pgfplots/xmax}{\xmax}
        \pgfkeysgetvalue{/pgfplots/ymin}{\ymin}
        \pgfkeysgetvalue{/pgfplots/ymax}{\ymax}

        \pgfmathsetmacro{\xArel}{#1}
        \pgfmathsetmacro{\yArel}{#3}
        \pgfmathsetmacro{\xBrel}{#1-#2}
        \pgfmathsetmacro{\yBrel}{\yArel}
        \pgfmathsetmacro{\xCrel}{\xArel}

        \pgfmathsetmacro{\lnxB}{\xmin*(1-(#1-#2))+\xmax*(#1-#2)} 
        \pgfmathsetmacro{\lnxA}{\xmin*(1-#1)+\xmax*#1} 
        \pgfmathsetmacro{\lnyA}{\ymin*(1-#3)+\ymax*#3} 
        \pgfmathsetmacro{\lnyC}{\lnyA+#4*(\lnxA-\lnxB)}
        \pgfmathsetmacro{\yCrel}{\lnyC-\ymin)/(\ymax-\ymin)} 

        \coordinate (A) at (rel axis cs:\xArel,\yArel);
        \coordinate (B) at (rel axis cs:\xBrel,\yBrel);
        \coordinate (C) at (rel axis cs:\xCrel,\yCrel);

        \draw[#5]   (A)-- node[pos=0.5,anchor=south] {#6 1}
                    (B)-- 
                    (C)-- node[pos=0.5,anchor=west] {#6 #4}
                    cycle;
    }
}
\newcommand{\logLogSlopeTrianglelow}[6]
{

    \pgfplotsextra
    {
        \pgfkeysgetvalue{/pgfplots/xmin}{\xmin}
        \pgfkeysgetvalue{/pgfplots/xmax}{\xmax}
        \pgfkeysgetvalue{/pgfplots/ymin}{\ymin}
        \pgfkeysgetvalue{/pgfplots/ymax}{\ymax}

        \pgfmathsetmacro{\xArel}{#1}
        \pgfmathsetmacro{\yArel}{#3}
        \pgfmathsetmacro{\xBrel}{#1-#2}
        \pgfmathsetmacro{\yBrel}{\yArel}
        \pgfmathsetmacro{\xCrel}{\xBrel}

        \pgfmathsetmacro{\lnxB}{\xmin*(1-(#1-#2))+\xmax*(#1-#2)} 
        \pgfmathsetmacro{\lnxA}{\xmin*(1-#1)+\xmax*#1} 
        \pgfmathsetmacro{\lnyA}{\ymin*(1-#3)+\ymax*#3} 
        \pgfmathsetmacro{\lnyC}{\lnyA-#4*(\lnxA-\lnxB)}
        \pgfmathsetmacro{\yCrel}{\lnyC-\ymin)/(\ymax-\ymin)} 

        \coordinate (A) at (rel axis cs:\xArel,\yArel);
        \coordinate (B) at (rel axis cs:\xBrel,\yBrel);
        \coordinate (C) at (rel axis cs:\xCrel,\yCrel);

        \draw[#5]   (A)-- node[pos=0.5,anchor=north] {#6 $1$}
                    (B)-- node[pos=0.5,anchor=east] {#6 $#4$}
                    (C)-- 
                    cycle;
    }
}

\title[Non-symmetric time dependent coupling]
       {On the non-symmetric coupling method for parabolic-elliptic interface problems} 

\author{Herbert Egger}
\address{TU Darmstadt, Department of Mathematics, Dolivostra\ss{}e 15, 64293 Darmstadt, Germany}
\email{egger@mathematik.tu-darmstadt.de}

\author{Christoph Erath}
\address{TU Darmstadt, Department of Mathematics, Dolivostra\ss{}e 15, 64293 Darmstadt, Germany}
\email{erath@mathematik.tu-darmstadt.de}

\author{Robert Schorr}
\address{TU Darmstadt, Graduate School of Computational Engineering, Dolivostra\ss{}e 15, 64293 Darmstadt, Germany}
\email{schorr@gsc.tu-darmstadt.de}

\thanks{H. Egger: TU Darmstadt, Germany; erath@mathematik.tu-darmstadt.de}
\thanks{C. Erath (corresponding author): TU Darmstadt, Germany; erath@mathematik.tu-darmstadt.de}
\thanks{R. Schorr: TU Darmstadt, Germany; schorr@gsc.tu-darmstadt.de;
The research of this author was supported by the \emph{Excellence Initiative} 
  of the German Federal and State Governments 
  and the \emph{Graduate School of Computational Engineering} at TU Darmstadt.}

\date{\bf\today}

\newcommand{\RR}{\mathbb{R}} 

\def\T{\mathcal{T}}  
\def\E{\mathcal{E}}  
\def\O{\mathcal{O}}  
\def\P{\mathcal{P}}  

\def\normal{n}

\def\V{\mathcal{V}} 
\def\K{\mathcal{K}} 
\def\S{\mathcal{S}} 
\def\R{\mathcal{R}} 

\def\diam{{\operatorname{diam}}}

\def\dt{\partial_t}

\def\dn{\partial_{\normal}}
\def\dtau{d_\tau}

\def\set#1#2{\left\{#1\,:\,#2\right\}}

\newcommand{\norm}[3][]{#1\|#2#1\|_{#3}}

\newcommand{\dual}[3][]{#1\langle#2\hspace*{.5mm},#3#1\rangle}
\newcommand{\product}[3][]{#1(#2\hspace*{.5mm},#3#1)}

\begin{document}

\begin{abstract}
We consider the numerical approximation of parabolic-elliptic interface problems 
by the non-symmetric coupling method of 
MacCamy and Suri [Quart. Appl. Math., 44 (1987), pp. 675--690]. 
We establish
well-posedness of this formulation for problems with non-smooth interfaces and 
prove quasi-optimality for a 
class of conforming Galerkin approximations in space. 
Therefore, 
error estimates with optimal order can be deduced for the 
semi-discretization in space by appropriate finite and boundary elements.
Moreover, we investigate the subsequent discretization in time by a variant of 
the implicit Euler method. 
As for the semi-discretization, we establish well-posedness and quasi-optimality 
for the fully discrete scheme
under minimal regularity assumptions on the solution. Error estimates with optimal order 
follow again directly.
Our analysis is based on estimates in appropriate energy norms. 
Thus, we do not use duality arguments and corresponding estimates 
for an elliptic projection which 
are not available for the non-symmetric coupling method. 
Additionally, we provide again error estimates under minimal 
regularity assumptions.
Some numerical examples illustrate our theoretical results.\\[0.25\baselineskip]
\end{abstract}

\maketitle

\section{Introduction} \label{sec:intro}
In this paper, we consider the numerical solution of parabolic-elliptic interface problems via 
the non-symmetric coupling method of MacCamy and Suri~\cite{MacCamy:1987}, which consists of a 
Galerkin approximation in space 
and a subsequent discretization in time by a variant of the implicit Euler method. 
For ease of presentation we consider the following simple model problem:
Find $u$ and $u_e$ such that 
\begin{alignat}{2}
\dt u -\Delta u &= \tilde f              &\qquad& \text{in }\Omega\times (0,T),   \label{eq:model1} \\
           -\Delta u_e &= 0              &\qquad& \text{in } \Omega_e \times (0,T) \label{eq:model2} \\
\intertext{with coupling conditions across the interface given by}
                     u &= u_e + \tilde g         &\qquad& \text{on } \Gamma \times (0,T), \label{eq:model3} \\
          \partial_{\normal} u &= \partial_{\normal} u_e +\tilde h\ &\qquad& \text{on } \Gamma \times (0,T) \label{eq:model4}.
\end{alignat}
For the presentation of our results we assume that
$\Omega \subset \RR^2$ is some bounded Lipschitz domain with $\diam(\Omega)<1$. However, all results also hold for three dimensions. 
We further denote by $\Gamma:=\partial \Omega$ and $\Omega_e = \RR^2 \setminus \overline{\Omega}$ the boundary
and the complement of $\Omega$,
and by $T>0$ a fixed end time. 
The co-normal derivative $\partial_{\normal} u = \nabla u \cdot \normal|_\Gamma$ is taken in direction 
of the unit normal vector $\normal$ on $\Gamma$ pointing outward with respect to $\Omega$.
The input data for the model are $\tilde f$, $\tilde g$, and $\tilde h$.  
To ensure the uniqueness of the solution, we additionally require the following initial and radiation conditions
\begin{alignat}{2}
            u(\cdot,0) &= 0               &\qquad& \text{on } \Omega, \label{eq:model5}\\
            u_e(x,t) &= a(t) \log|x| + \O(|x|^{-1})  &\qquad& |x| \to \infty. \label{eq:model6}
\end{alignat}
The function $a(t):[0,T]\to\RR$ is unknown and automatically determined
in the solving process, see \cref{rem:solution}.  
A system of this type arises, for instance, in the modeling of eddy currents in the 
magneto-quasistatic regime~\cite{MacCamy:1987}.
In our model problem we might also allow inhomogeneous initial data 
and extra Dirichlet or Neumann boundaries in the interior domain.
Then the analysis in this paper holds by obvious modifications.  

Using the well-known representation formula~\cite{McLean:2000-book}, 
the field $u_e$ in the exterior domain can be expressed via the traces $u_e|_\Gamma$ 
and $\phi:=\partial_{\normal} u_e|_\Gamma$ on the interface $\Gamma$.
This allows us to reduce the above problem to a parabolic partial differential 
equation in $\Omega$ coupled to an integral equation at the boundary $\Gamma$ 
with $u$ and $\phi$ as the unknown fields. 
Different equivalent formulations are possible here, which 
lead, after discretization, to various numerical approximation schemes.
Based on the non-symmetric coupling method of Johnson and N{\'e}d{\'e}lec~\cite{Johnson:1980-1}, 
MacCamy and Suri~\cite{MacCamy:1987} established the well-posedness of 
problem~\cref{eq:model1}--\cref{eq:model6} via the method of Galerkin approximation.
Their analysis is based on the compactness of the double layer operator which 
relies on the assumption that $\Gamma$ is smooth~\cite{Costabel:1988-1}. 
As a by-product of their analysis, the authors also proved quasi-optimal error estimates
in the energy norm
for general Galerkin approximations under mild assumptions on the approximation spaces, i.e., 
\begin{align*}
\norm{u - u_h&}{L^2(0,T;H^1(\Omega))} + \norm{\dt u - \dt u_h}{L^2(0,T;H^1(\Omega)')} + 
\norm{\phi - \phi_h}{L^2(0,T;H^{-1/2}(\Gamma))}\\
&\le C \inf_{v_h,\psi_h} \{\norm{u - v_h}{L^2(0,T;H^1(\Omega))} + \norm{\dt u - \dt v_h}{L^2(0,T;H^1(\Omega)')} \\
&\qquad\qquad\quad+ \norm{\phi - \psi_h}{L^2(0,T;H^{-1/2}(\Gamma))}\}.
\end{align*}
Here $u_h$ and $\phi_h$ are the semi-discrete approximations of $u$ and $\phi$, respectively.
Hence, a discretization by appropriate finite and boundary elements directly leads to error estimates 
with optimal order for the resulting semi-discrete schemes.

To overcome the restrictive smoothness assumption on the domain $\Omega$, 
Costabel, Ervin, and Stephan~\cite{Costabel:1990} applied the symmetric coupling approach 
proposed in~\cite{Costabel:1988-2} to treat the parabolic-elliptic interface problem 
stated above. This allowed them to prove the well-posedness 
of~\cref{eq:model1}--\cref{eq:model6} and the quasi-optimality of 
Galerkin approximations also for non-smooth domains.
In addition, they investigated the subsequent time discretization by the Crank-Nicolson method and 
established error estimates for the resulting fully discrete scheme.
The analysis of~\cite{Costabel:1990} is based on an 
elliptic projection and corresponding error estimates in $L^2$, and therefore 
relies on duality arguments; see e.g.~\cite{Varga:1971-book,Wheeler:1973}.  
Due to a lack of ``adjoint consistency'' for the non-symmetric coupling method of MacCamy and Suri
these arguments cannot be used for its analysis.
Therefore, ``an analysis of a fully discretized version of their coupling scheme is not available and will 
be difficult'', as argued in~\cite{Costabel:1990}.

In this paper, we close this gap in the analysis of the non-symmetric coupling 
method for parabolic-elliptic interface problems.  
Our main results can be summarized as follows: 
\begin{itemize}
\item Based on an argument of Sayas~\cite{Sayas:2009-1}, Steinbach~\cite{Steinbach:2011} 
showed that the non-symmetric coupling of the elliptic-elliptic interface problem 
with a lowest order term in the interior domain  
in fact leads to a coercive variational formulation; see also \cite{Erath:2017-1}. 
This allows us to extend the results of~\cite{MacCamy:1987,Costabel:1990} to the non-symmetric 
coupling method on non-smooth domains. In particular, we establish well-posedness 
of this formulation and prove quasi-optimal error estimates for Galerkin approximations.
\item As a second step of our analysis, we also consider the time discretization of the 
semi-discrete scheme of~\cite{MacCamy:1987} by a variant of the implicit Euler method.
We utilize a formulation that is fully consistent with the continuous variational 
formulation and does not require additional smoothness of the solution or the data;
see~\cite{Tantardini:2014-1} for a related approach in the context of parabolic problems. 
This allows us to establish well-posedness and quasi-optimal approximation properties 
with respect to the energy norm under minimal smoothness assumptions on the solution.
\end{itemize}

For ease of notation, we will present the details of our analysis only for 
the simple model problem~\cref{eq:model1}--\cref{eq:model6} stated above.
Our arguments, however, are quite general and can be also applied to interface problems with more general 
parabolic operators and interface conditions, and in higher space dimensions. 
Our approach might also be useful for the analysis of other coupling strategies; 
let us refer to \cite{Aurada:2013-1} for a recent survey of possible couplings.

The remainder of the manuscript is organized as follows:
In \cref{sec:prelim}, we introduce our basic notation and assumptions. 
Then we present the weak formulation of the non-symmetric coupling approach 
and establish its well-posedness.
\Cref{sec:galerkin} introduces a semi-discretization of the variational 
problem in space by a Galerkin approach. 
Furthermore, we establish well-posedness of the semi-discrete scheme and 
quasi-optimal approximation properties. 
In \cref{sec:time}, we discuss the time discretization by a variant of the implicit 
Euler method 
and prove again quasi-optimal error estimates under minimal smoothness assumptions. 
In \cref{sec:fembem}, we consider space discretization by finite and boundary
elements. Using the analysis of the previous sections, 
we derive explicit error estimates for the resulting 
semi-discrete and fully-discrete schemes. 
For illustration of our theoretical results, we present
some numerical tests in \cref{sec:numerics}.

\section{Notation and weak formulation} \label{sec:prelim}

In this section, we first introduce some basic notation and assumptions.
Then we formulate and analyze a weak formulation of our model problem. 

\subsection{Notation and basic assumptions}

Throughout the next sections, we make the following assumption on the domain:
\begin{align}
 \label{as:A1}
 \tag{A1}
  \Omega \subset \RR^2 \text{ is a bounded Lipschitz domain and }\diam(\Omega)<1.
\end{align}
Note that $\diam(\Omega)<1$ can always be achieved by scaling.
We write $H^s(\Omega)$ and $H^s(\Gamma)$ for the usual Sobolev spaces
and denote by $H^s(\Omega)'$ and $H^{-s}(\Gamma)=H^s(\Gamma)'$ 
their dual spaces with respect to the duality pairing induced by $L^2$; see~\cite{Evans:2010-book,McLean:2000-book} for details. 
We use $\product{\cdot}{\cdot}_\Omega$ and $\dual{\cdot}{\cdot}_\Omega$, and on the 
boundary $\product{\cdot}{\cdot}_{\Gamma}$ and $\dual{\cdot}{\cdot}_{\Gamma}$ 
to denote the corresponding scalar products and duality pairings.
Let us recall that 
\begin{align*}
\dual{\psi}{v}_{\Gamma} 
\le \norm{\psi}{H^{-1/2}(\Omega)} \norm{v}{H^{1/2}(\Gamma)} 
\le C_{tr}  \norm{\psi}{H^{-1/2}(\Omega)} \norm{v}{H^1(\Omega)}
\end{align*}
for all $\psi \in H^{-1/2}(\Gamma)$ and $v \in H^1(\Omega)$ with a constant $C_{tr}>0$. 
In the first and second statement, one should formally write $\gamma v$ instead of $v$, 
where $\gamma : H^1(\Omega) \to H^{1/2}(\Gamma)$ denotes the trace operator. 
We skip the explicit notation of the trace operator since the meaning is clear from the context.
The last inequality encodes the continuity of the trace operator. 

For ease of presentation and to allow for an easy comparison of the results, 
we adopt the notation of~\cite{Costabel:1990} and denote by
\begin{align*}
H&=H^1(\Omega) \qquad \text{and} \qquad B=H^{-1/2}(\Gamma)
\end{align*}
the main function spaces arising in our analysis.
Furthermore, we use 
\begin{align*}
H_T=L^2(0,T;H) \qquad \text{and} \qquad B_T=L^2(0,T;B)
\end{align*}
to denote the corresponding Bochner spaces of functions on $[0,T]$ with values in $H$ and $B$, respectively. 
The associated dual spaces are given by $H'=H^1(\Omega)'$ and $B'=H^{-1/2}(\Gamma)'=H^{1/2}(\Gamma)$ 
as well as $H_T'=L^2(0,T;H')$ and $B_T'=L^2(0,T;B')$. All spaces introduced above are Hilbert spaces if equipped 
with their natural norms, e.g., $\norm{u}{H_T}^2 = \int_0^T \norm{u(t)}{H}^2\, dt$. We further use
\begin{align*}
Q_T = \set{u \in H_T}{\dt u \in H_T' \text{ and } u(0)=0}
\end{align*}
to denote the natural energy space for the parabolic problem with the norm
\begin{align*}
\norm{u}{Q_T}^2 := \norm{u}{H_T}^2 + \norm{\dt u}{H_T'}^2. 
\end{align*}
This space is again complete.
It is well-known that the space $Q_T$ is continuously 
embedded in $C([0,T];L^2(\Omega))$; see, e.g.,~\cite{Evans:2010-book}.
Thus the initial value $u(0)=0$ makes sense.

\subsection{Preliminaries}

Let $(u,u_e)$ denote a sufficiently smooth solution of problem~\cref{eq:model1}--\cref{eq:model6}.
Then multiplying equation~\cref{eq:model1} with a test function $v \in H^1(\Omega)$, 
integrating over $\Omega$, and using integration by parts formally lead to 
\begin{align*}
\int_\Omega \dt u(t) v\,dx  + \int_\Omega \nabla u(t)\cdot\nabla v\,dx 
- \int_{\Gamma} \phi(t) v\,ds 
= \int_\Omega \tilde f(t) v\,dx + \int_\Gamma \tilde h v\,ds.
\end{align*}
Here, we used equation~\cref{eq:model4} with $\phi:=\dn u_e|_\Gamma$ 
to replace the interior co-normal derivative.
For the right-hand side, we will use the short hand notation 
\begin{align}
 \label{eq:f}
 \dual{f}{v}_{\Omega} := \int_\Omega \tilde f v \,dx + \int_{\Gamma} \tilde h v \,ds. 
\end{align}
and write $f\in H'_T$.
With the representation formula for the Laplacian, we can further express 
the solution for~\cref{eq:model2} and~\cref{eq:model6} in the exterior domain 
$\Omega_e$ by
\begin{align}
 \label{eq:repformular}
u_e(x) = \int_{\Gamma} \partial_{\normal_y} G(x,y) u_e(y)|_\Gamma\,ds_y 
- \int_{\Gamma} G(x,y) \dn u_e(y)|_\Gamma\,ds_y.
\end{align}
Here $G(x,y) = - \frac{1}{2\pi} \log|x-y|$ denotes the fundamental solution 
of the Laplace operator in two dimensions~\cite{McLean:2000-book}. 
Upon taking the trace at the boundary $\Gamma$, writing again $\phi=\dn u_e|_\Gamma$ at $\Gamma$, 
and using the coupling condition~\cref{eq:model3} 
to replace $u_e|_\Gamma$ by $u|_\Gamma$ 
we obtain 
\begin{align}
 \label{eq:g}
\V \phi + (1/2-\K) u|_{\Gamma} =  (1/2 - \K) \tilde g =:g. 
\end{align}
Here, $\V$ and $\K$ denote the single and double layer operators.
For sufficiently smooth functions and domains they are given by~\cite{McLean:2000-book}
\begin{align*}
(\V \psi)(x) = \int_{\Gamma} G(x,y) \psi(y) \,ds_y 
\qquad \text{and} \qquad 
(\K v)(x) = \int_{\Gamma} \partial_{\normal_y} G(x,y) v(y) \,ds_y.
\end{align*}
By assumption~\cref{as:A1} they can be extended to 
bounded linear operators on $H^{-1/2}(\Gamma)$ 
and $H^{1/2}(\Gamma)$, respectively; see~\cref{lem:elliptic}.

\subsection{Variational formulation}

A combination of the above formulas leads to the following weak formulation, 
which will be the starting point for our analysis.

\begin{problem}[Variational problem] \label{prob:variational}
Given $f \in H'_T$ and $g \in B_T'$, find $u \in Q_T$ and $\phi \in B_T$ such that 
\begin{align}
\dual{\dt u(t)}{v}_\Omega + \product{\nabla u(t)}{\nabla v}_\Omega - \dual{\phi(t)}{v}_{\Gamma} 
&= \dual{f(t)}{v}_\Omega, \label{eq:vp1}\\
\dual{(1/2-\K) u(t)|_\Gamma}{\psi}_{\Gamma} + \dual{\V \phi(t)}{\psi}_{\Gamma} 
&= \dual{g(t)}{\psi}_{\Gamma}  \label{eq:vp2}
\end{align}
for all test functions $v \in H=H^1(\Omega)$ and $\psi \in B=H^{-1/2}(\Gamma)$, 
and for a.e. $t \in [0,T]$. 
\end{problem}
\begin{remark} 
 \label{rem:solution}
Any sufficiently smooth solution of~\cref{eq:model1}--\cref{eq:model6} 
also solves~\cref{eq:vp1}--\cref{eq:vp2} with 
$\dual{f}{v}_{\Omega} = \dual{\tilde f}{v}_\Omega + \dual{\tilde h}{v}_\Gamma$ and
$\dual{g}{\psi}_\Gamma=\dual{(1/2-\K)\tilde g}{\psi}_\Gamma$ and, vice versa, 
any regular solution $(u,\phi)$ of~\cref{eq:vp1}--\cref{eq:vp2}
is a classical solution of~\cref{eq:model1}--\cref{eq:model6}. 
We note that $a(t)$ in~\cref{eq:model6}
can be expressed directly in terms of the field $u_e$, once the solution $(u,\phi)$
of~\cref{eq:vp1}--\cref{eq:vp2} is known, i.e.,
$a(t)=\frac{1}{2\pi} \int_{\Gamma} \phi \,ds$, where $\phi=\partial_{\normal} u_e|_\Gamma$.
\end{remark}
The analysis of \cref{prob:variational} is based on the following auxiliary results. 
\begin{lemma} \label{lem:elliptic}
Let~\cref{as:A1} hold. 
Then the linear operators $\V:H^{s-1/2}(\Gamma) \to H^{s+1/2}(\Gamma)$ 
and $\K : H^{s+1/2}(\Gamma) \to H^{s+1/2}(\Gamma)$, $s\in [-1/2,1/2]$, are bounded  
and $\V$ is elliptic on $H^{-1/2}(\Gamma)$, i.e., 
\begin{align*}
\dual{\V \psi}{\psi}_{\Gamma} \ge C_\V \norm{\psi}{H^{-1/2}(\Gamma)}^2 
\qquad \text{for all } \psi \in H^{-1/2}(\Gamma)
\end{align*}
with some $C_\V>0$ independent of $\psi$. Moreover, the bilinear form
\begin{align*}
a(u,\phi;v,\psi) := \product{\nabla u}{\nabla v}_\Omega - \dual{\phi}{v}_{\Gamma} 
+ \dual{(1/2 - \K) u}{\psi}_{\Gamma} + \dual{\V \phi}{\psi}_{\Gamma},
\end{align*}
is continuous and satisfies a G\r{a}rding inequality 
on $H^1(\Omega) \times H^{-1/2}(\Gamma)$, i.e., 
\begin{align*}
a(v,\psi;v,\psi) + \product{v}{v}_\Omega \ge \alpha \big(\norm{v}{H^1(\Omega)}^2 + 
\norm{\psi}{H^{-1/2}(\Gamma)}^2 \big)
\end{align*}
with $\alpha>0$ independent of the functions $v \in H^1(\Omega)$ and $\psi \in H^{-1/2}(\Gamma)$.
\end{lemma}
\begin{proof}
Boundedness and ellipticity of the integral operators are well-known; 
see for instance~\cite{Costabel:1988-1,McLean:2000-book}. 
The coercivity estimate for the bilinear form $a(\cdot;\cdot)$, on the other hand, follows
directly by applying~\cite[Theorem~1]{Erath:2017-1} 
with $\mathbf{A}=\mathcal{I}$, $C_{\mathbf{b}c}=1$, and $\beta=0$.
\end{proof}

Using these properties, we now prove the well-posedness of \cref{prob:variational}.
\begin{theorem} 
\label{thm:wellposed}
Let~\cref{as:A1} hold. Then for any $f \in H_T'$ and $g \in B_T'$, \cref{prob:variational} 
admits a unique weak solution
$(u,\phi) \in Q_T \times B_T$ and 
\begin{align*}
\norm{u}{Q_T} + \norm{\phi}{B_T} \le C ( \norm{f}{H_T'} + \norm{g}{B_T'})
\end{align*}
with a constant $C>0$ that only depends on the domain $\Omega$ and the time horizon $T$.
\end{theorem}
\begin{proof}
Since $\V$ is elliptic and thus invertible, we can use~\cref{eq:vp2} to express
$\phi(t) = \S u(t) + \R g(t)$ with $\S=\V^{-1} (\K - 1/2)$ and $\R=\V^{-1}$. 
Then~\cref{eq:vp1} can be reduced to 
\begin{align} \label{eq:reduced}
\dual{\dt u(t)}{v}_\Omega + \tilde a(u(t),v) = \dual{f(t)}{v}_{\Omega} + \dual{\R g(t)}{v}_{\Gamma}
\end{align}
with the bilinear form $\tilde a(u,v) := \product{\nabla u}{\nabla v}_\Omega - \dual{\S u}{v}_{\Gamma}$. 
From the G\r{a}rding inequality for the bilinear form $a(\cdot,\cdot)$ 
in \cref{lem:elliptic} with 
$\psi = \V^{-1}(\K-1/2) v$ we deduce that for all $v\in H^1(\Omega)$
\begin{align*}
\tilde a(v,v) + \product{v}{v}_\Omega 
&= a(v,\psi;v,\psi) + \product{v}{v}_\Omega 
 \ge \alpha \norm{v}{{H^1(\Omega)}}^2. 
\end{align*}
Thus $\tilde a(u,v)$ satisfies the G\r{a}rding inequality on $H^1(\Omega)$.
Consequently, the reduced problem~\cref{eq:reduced} is uniformly parabolic. 
The assertions for $u$ in~\cref{eq:reduced} then follow from standard results about variational 
evolution problems, see, e.g.,~\cite[Ch.~XVIII, Par.~3]{Dautray:1992-5}
and~\cite[Part II, Sec. 7.1.2]{Evans:2010-book}.
To bound the second solution component $\phi$
we use~\cref{eq:vp2} and the ellipticity of $\V$ which gives
\begin{align*}
C_\V \norm{\phi(t)}{{H^{-1/2}(\Gamma)}}^2 
&\le \dual{V \phi(t)}{\phi(t)}_{\Gamma}
 =-\dual{(1/2-\K) u(t)}{\phi(t)}_{\Gamma} + \dual{g(t)}{\phi(t)}_{\Gamma} \\
&\le \big((1/2+C_\K) C_{tr} \norm{u(t)}{H^1(\Omega)} 
+ \norm{g(t)}{H^{1/2}(\Gamma)} \big) \norm{\phi(t)}{H^{-1/2}(\Gamma)}.  
\end{align*}
In the last step, we used the trace inequality and 
the boundedness of $\K$. 
\end{proof}

\begin{corollary}
 For $\tilde f \in H'_T$, $\tilde g\in B'_T$, and $\tilde h\in B_T$ our model 
 problem~\cref{eq:model1}--\cref{eq:model6} admits a unique weak solution $(u,\phi) \in Q_T \times B_T$
 and
 \begin{align*}
  \norm{u}{Q_T} + \norm{\phi}{B_T}\le C( \norm{\tilde f}{H_T'} + \norm{\tilde h}{B_T} + \norm{\tilde g}{B_T'}).
 \end{align*}
\end{corollary}
\begin{proof}
 This follows directly from \cref{thm:wellposed} with~\cref{eq:f} and \cref{eq:g}.
\end{proof}

\section{Galerkin approximation} \label{sec:galerkin}
Let $H^h \subset H^1(\Omega)$ and $B^h \subset H^{-1/2}(\Omega)$ be finite dimensional 
subspaces. Similar as before, we define corresponding Bochner spaces $H^h_T = L^2(0,T;H^h)$ 
and $B_T^h = L^2(0,T;B^h)$ and
the corresponding energy space is denoted by $Q^h_T = \set{v_h \in H^1(0,T;H^h)}{v_h(0)=0}$. 
Then we consider the following Galerkin approximation of \cref{prob:variational}.
\begin{problem}[Semi-discrete problem] 
\label{prob:semidiscrete}
Find $u_h \in Q^h_T$ and $\phi_h \in B^h_T$ such that 
\begin{align}
\product{\dt u_h(t)}{v_h}_\Omega + \product{\nabla u_h(t)}{\nabla v_h}_\Omega 
- \product{\phi_h(t)}{v_h}_{\Gamma} &= \dual{f(t)}{v_h}_\Omega \label{eq:vp1h}\\
\product{(1/2-\K) u_h(t)}{\psi_h}_{\Gamma} + \product{\V \phi_h(t)}{\psi_h}_{\Gamma} 
&= \product{g(t)}{\psi_h}_{\Gamma} \label{eq:vp2h}
\end{align}
for all test functions $v_h \in H^h$ and $\psi_h \in B^h$, and for a.e. $t \in [0,T]$. 
\end{problem}
The analysis of this Galerkin approximation can be carried out with similar arguments as used 
in~\cite{Costabel:1990} and~\cite{MacCamy:1987}. Hence we make use of \cref{lem:elliptic}
to get rid of the smoothness assumption on $\Gamma$.
For convenience of the reader and later reference, 
we briefly state the main results and sketch the basic ideas of their proofs.
Due to \cref{lem:elliptic}, the well-posedness of the above problem follows 
again by standard energy arguments.
\begin{lemma} \label{lem:wellposedh}
Let~\cref{as:A1} hold. Then \cref{prob:semidiscrete} has a unique solution.
Moreover, 
\begin{align}
 \label{eq:energyh}
\norm{u_h}{H_T} + \norm{\phi_h}{B_T} \le C \big( \norm{f}{H_T'} + \norm{g}{B_T'} \big)
\end{align}
with a constant $C>0$ that is independent of the data $f$, $g$ and the spaces $H^h,B^h$.
\end{lemma}
\begin{proof}
We proceed with similar arguments as in the proof of \cref{thm:wellposed}:
First, we use~\cref{eq:vp2h} to express $\phi_h(t) = \S_h u_h(t) + \R_h g(t)$, 
where $\S_h:H^h \to B^h$ is defined by 
\begin{align}
 \label{eq:Sh}
\dual{\V \S_h u_h}{\psi_h}_{\Gamma} 
= \dual{(\K-1/2) u_h}{\psi_h}_{\Gamma} \qquad \text{for all } \psi_h \in B^h. 
\end{align}
and $\R_h : H^{1/2}(\Gamma) \to B^h$ is defined by 
\begin{align}
 \label{eq:Rh}
\dual{\V \R_h g}{\psi_h}_{\Gamma} = \dual{g}{\psi_h}_{\Gamma} \qquad \text{for all } \psi_h \in B^h.  
\end{align}
Due to the Lax-Milgram Lemma both equations~\cref{eq:Sh}--\cref{eq:Rh}
have unique solutions since 
$\V:H^{-1/2}(\Gamma)\to H^{1/2}(\Gamma)$ is bounded and elliptic,
and $H^h\subset H$ and $B^h\subset B$ are finite dimensional and thus complete subspaces. 
Hence, $\S_h$ and $\R_h$ are well-defined.
Furthermore, it directly follows that $\norm{\R_h g}{H^{-1/2}(\Gamma)} \le C_\V^{-1} \norm{g}{H^{1/2}(\Gamma)}$. 
Then~\cref{eq:vp1h} can again be reduced to an ordinary differential equation
\begin{align}
 \label{eq:reducedproblemh} 
\product{\dt u_h(t)}{v_h}_\Omega + \tilde a_h(u_h(t),v_h) = \dual{f(t)}{v_h}_\Omega 
+ \product{\R_h g(t)}{v_h}_{\Gamma} 
\end{align}
with bilinear form $\tilde a_h(u_h,v_h) = \product{\nabla u_h}{\nabla v_h}_\Omega 
- \product{\S_h u_h}{v_h}_{\Gamma}$. Using \cref{lem:elliptic} with $u=v=u_h$ and
$\phi=\psi=\psi_h=\S_h u_h$, where $\S_h$ is defined by~\cref{eq:Sh}, we  
obtain for all $u_h\in H^h$ that  
\begin{align}
 \label{eq:ahgarding}
\tilde a_h(u_h,u_h) + \product{u_h}{u_h}_\Omega 
= a(u_h,\psi_h;u_h,\psi_h)+ \product{u_h}{u_h}_\Omega 
\ge \alpha \norm{u_h}{H^1(\Omega)}^2. 
\end{align}
Existence and uniqueness of a solution to the reduced problem~\cref{eq:reducedproblemh} 
and the estimates for $\norm{u_h}{H_T}$ can again be obtained from the abstract 
results of~\cite{Dautray:1992-5,Evans:2010-book}. 
To estimate $\norm{\phi_h(t)}{H^{-1/2}(\Gamma)}$ 
we use~\cref{eq:vp2h} and the same arguments as in the proof of \cref{thm:wellposed} and get
\begin{align}
 \label{eq:energyphih}
 C_\V\norm{\phi_h(t)}{H^{-1/2}(\Gamma)}
 \le (1/2+C_\K) C_{tr} \norm{u_h(t)}{H^1(\Omega)} 
+ \norm{g(t)}{H^{1/2}(\Gamma)}.    
\end{align}
\end{proof}
In order to obtain a uniform estimate also for the time derivative $\dt u_h$, 
which is not included in~\cref{eq:energyh}, we proceed with similar arguments 
as~\cite{MacCamy:1987,Costabel:1990}. 
Let $P_h : L^2(\Omega) \to H^h$ denote the $L^2$-orthogonal projection defined by
\begin{align}\label{eq:projection}
\product{P_h v}{w_h}_\Omega = \product{v}{w_h}_\Omega \qquad \text{for all } w_h \in H^h.
\end{align} 
We will assume that the $L^2$-projection $P_h$ is stable in $H^1(\Omega)$, i.e., 
there exists a constant $C_P>0$ such that 
\begin{align}
 \label{as:A2}
 \tag{A2}
 \norm{P_h v}{H^1(\Omega)} \le C_P \norm{v}{H^1(\Omega)} \text{ for all }v \in H^1(\Omega).
\end{align}
This imposes a mild condition on the approximation 
space $H^h$, which is not very restrictive in practice; see \cref{sec:fembem} for an example.
Property (A2) and equation~\cref{eq:vp1h} can now be used 
to deduce a uniform bound for the norm $\norm{\dt u_h}{H^1(\Omega)'}$ of the time derivative
and the following energy estimate.
\begin{lemma}[Discrete energy estimate] \label{lem:energyh} 
Let~\cref{as:A1}--\cref{as:A2} hold. Then 
\begin{align*}
\norm{u_h}{Q_T} + \norm{\phi_h}{B_T} \le C \big( \norm{f}{H_T'} + \norm{g}{B_T'} \big)
\end{align*}
with a constant $C>0$ independent of $f,g$ and the approximation spaces $H^h$ and $B^h$.
\end{lemma}
\begin{proof}
By definition of the dual norm and the $L^2$-projection, we obtain
\begin{align}
 \label{eq:dualnorm}
\norm{\dt u_h(t)}{H^{1}(\Omega)'} 
&= \sup_{0\not= v \in H^1(\Omega)} \frac{\product{\dt u_h(t)}{v}_\Omega}{\norm{v}{H^1(\Omega)}}
 = \sup_{0\not= v \in H^1(\Omega)} \frac{\product{\dt u_h(t)}{P_h v}_\Omega}{\norm{v}{H^1(\Omega)}}.
\end{align}
Using equation~\cref{eq:vp1h}, the Cauchy-Schwarz inequality, and the trace inequality, 
one can further estimate 
\begin{align*}
\product{\dt u_h(t)}{P_h v}_\Omega \le \big(\norm{u_h(t)}{H^1(\Omega)} 
+ C_{tr} \norm{\phi_h(t)}{H^{-1/2}(\Omega)} + \norm{f(t)}{H^{1}(\Omega)'} \big) \norm{P_h v}{H^1(\Omega)}.
\end{align*}
Therefore, assumption~\cref{as:A2} yields 
\begin{align*}
\norm{\dt u_h(t)}{H^{1}(\Omega)'} 
 \le C \big( \norm{u_h(t)}{H^1(\Omega)} + \norm{\phi_h(t)}{H^{-1/2}(\Gamma)} 
 + \norm{f(t)}{H^{1}(\Omega)'} \big).
\end{align*}
Then the assertion of the lemma 
follows by integration over time and combination with the 
estimates~\cref{eq:energyh} for $\norm{u_h}{H_T}$ and $\norm{\phi_h}{B_T}$ 
stated in \cref{lem:wellposedh}.
\end{proof}
By combination of the previous lemmas and the variational problems defining 
the continuous and the semi-discrete solution, we now obtain the following result.
\begin{theorem}[Quasi-best-approximation] 
 \label{thm:quasioptimality} 
Let~\cref{as:A1}--\cref{as:A2} hold. Furthermore, $(u,\phi) \in Q_T \times B_T$ 
and $(u_{h},\phi_{h}) \in Q_T^{h} \times B_T^{h}$ denote the solutions of
\cref{prob:variational} and \cref{prob:semidiscrete}, respectively.
Then there holds that
\begin{align*}
\norm{u - u_h}{Q_T} + \norm{\phi-\phi_h}{B_T} \le C \big( \norm{u - \tilde u_h}{Q_T} 
+ \norm{\phi-\tilde \phi_h}{B_T} \big)
\end{align*}
for all functions $\tilde u_h \in Q_T^h$ and $\tilde \phi_h \in B_T^h$ with a
constant $C>0$ which is independent of the problem data $f,g$ and of the spaces $H^h$ and $B^h$.
\end{theorem}
\begin{proof}
This result was first proven in~\cite{Costabel:1990} 
for the symmetric coupling method. 
Using \cref{lem:elliptic},
their proof can be adopted to the non-symmetric coupling as well. 
For convenience of the reader and later reference, we only repeat the main arguments:
Let $\tilde u_h \in Q_T^h$ and $\tilde \phi_h \in B_T^h$ be arbitrary. 
By
\begin{align*}
\norm{u - u_h}{Q_T} &\le  \norm{u - \tilde u_h}{Q_T} + \norm{\tilde u_h - u_h}{Q_T} \qquad \text{and}\\
\norm{\phi - \phi_h}{B_T} &\le \norm{\phi - \tilde \phi_h}{B_T} + \norm{\tilde \phi_h - \phi_h}{B_T}
\end{align*}
we split the error into an \emph{approximation error} and a \emph{discrete error} component. 
The first part already appears in the final estimate.
To estimate the discrete error components we note that the discrete problem~\cref{eq:vp1h}--\cref{eq:vp2h} 
is consistent with the continuous problem~\cref{eq:vp1}--\cref{eq:vp2}.
Hence, we may write the discrete error components 
$w_h = \tilde u_h - u_h$ and $\rho_h = \tilde \phi_h - \phi_h$ as the solution of the system
\begin{align}
 \label{eq:discerror1}
\product{\dt w_h(t)}{v_h}_\Omega + \product{\nabla w_h(t)}{\nabla v_h}_\Omega 
- \product{\rho_h(t)}{v_h}_{\Gamma} &= \product{F(t)}{v_h}_\Omega\\
 \label{eq:discerror2}
\product{(1/2-\K) w_h(t)}{\psi_h}_{\Gamma} + \product{\V \rho_h(t)}{\psi_h}_{\Gamma} 
&= \dual{G(t)}{\psi_h}_{\Gamma}
\end{align}
for all $v_h\in H^h$ and $\psi_h\in B^h$ with the right-hand sides $F(t)$ and $G(t)$ defined by 
\begin{align*}
\dual{F(t)}{v}_{\Omega} 
  &:= \dual{\dt \tilde u_h(t) - \dt u(t)}{v}_\Omega 
  + \product{\nabla \tilde u_h(t) - \nabla u(t)}{\nabla v}_\Omega 
  - \dual{\tilde \phi_h(t) - \phi(t)}{v}_{\Gamma},\\
\dual{G(t)}{\psi}_{\Gamma} 
  &:= \dual{(1/2-\K) (\tilde u_h(t) - u(t))}{\psi}_{\Gamma} 
  + \dual{\V (\tilde \phi_h(t) - \phi(t))}{\psi}_{\Gamma}.
\end{align*}
for all $v\in H$ and $\psi\in B$.
With the bounds from the integral and trace operators, 
the Cauchy-Schwarz inequality, and integrating with respect to time, one can see that
\begin{align*}
\norm{F}{L^2(0,T;H^1(\Omega)')} 
&\le  C \big( \norm{u - \tilde u_h}{Q_T} + \norm{\phi - \tilde \phi_h}{B_T}\big)  \\
\norm{G}{L^2(0,T;H^{1/2}(\Gamma))} 
& \le C \big( \norm{u - \tilde u_h}{H_T} + \norm{\phi - \tilde \phi_h}{B_T}\big).
\end{align*}
Note that the system~\cref{eq:discerror1}--\cref{eq:discerror2} with the right-hand sides $F$ and $G$ 
has the same form as~\cref{eq:vp1h}--\cref{eq:vp2h}.
Therefore, \cref{lem:energyh} applies and finally shows that
\begin{align*}
\norm{\tilde u_h - u_h}{Q_T} + \norm{\tilde \phi_h - \phi_h}{B_T} 
\le C \big( \norm{u - \tilde u_h}{Q_T} + \norm{\phi - \tilde \phi_h}{B_T} \big).
\end{align*}
Together with the error splitting this completes the proof.
\end{proof}
\begin{remark} 
As a direct consequence of \cref{thm:quasioptimality}, we also obtain 
\begin{align*}
\norm{u - u_h}{Q_T} + \norm{\phi-\phi_h}{B_T} \le C \big( \norm{u - P_h u}{Q_T} 
+ \norm{\phi-\Pi_h \phi}{B_T} \big), 
\end{align*}
where $P_h:H^1(\Omega)\to H^h$ is the  $L^2(\Omega)$ projection operator introduced 
in~\cref{eq:projection}, $\Pi_h:H^{-1/2}(\Gamma)\to B_h$ is the $H^{-1/2}(\Gamma)$-projection 
operator, and $C>0$.
This allows us to obtain explicit error bounds for particular choices of 
approximation spaces by using interpolation error estimates in the energy spaces; 
see \cref{sec:fembem} for an example.
\end{remark}

\section{Time discretization} \label{sec:time}

For the time discretization of the Galerkin approximation,
we consider a particular one-step method that allows us to 
establish quasi-optimality of a fully discrete scheme under minimal regularity assumptions.
Let us note that a similar method was used in~\cite[Sec. 4.1.]{Tantardini:2014-1} for the discretization
of a parabolic problem.
First of all, we introduce some notation which we need to formulate our time discretization scheme.
Let $0=t^0 < t^1 < \ldots < t^N = T$, $N\in\mathbb{N}$ be a partition of the time interval $[0,T]$. 
Further, we denote by $\tau^n = t^n-t^{n-1}$ the local time step sizes and 
set $\tau := \max_{n=1,\ldots,N } \tau^{n}$.

In this section we search for approximations $u_{h,\tau} \in Q_T^{h,\tau}$
and $\phi_{h,\tau} \in B_T^{h,\tau}$ with 
\begin{align*}
Q_T^{h,\tau} 
&:=\set{u  \in C(0,T;H^h)}{u(0) = 0, u|_{[t^{n-1},t^n]} \text{ is linear in } t} \qquad \text{and} \\
B_T^{h,\tau} 
&:= \set{\phi  \in L^2(0,T;B^h)}{\phi|_{(t^{n-1},t^n]} \text{ is constant in } t}.
\end{align*}
Furthermore, for sufficiently regular functions in $t$, we denote by $v^n = v(t^n)$ 
the values at the grid points.  
For $u_{h,\tau}\in Q_T^{h,\tau}$ the operator $\dt$ has to be understood
piecewise with respect to the time mesh, in particular,
\begin{align}
 \label{eq:discderivative}
 \dt u_{h,\tau}|_{(t^{n-1},t^n)}=\dtau u_{h,\tau}^n
 \qquad\text{with}\qquad\dtau u_{h,\tau}^n := \frac{1}{\tau^n} (u_{h,\tau}^n - u_{h,\tau}^{n-1}).
\end{align}
We further introduce weighted averages
\begin{align} \label{eq:hatv}
\widehat{v}^n = \frac{1}{\tau^n} \int_{t^{n-1}}^{t^n} v(t) \omega^n(t)\,dt
\qquad\text{with }\omega^n(t)=\frac{6t-2t^n-4t^{n-1}}{\tau^n}
\end{align}
and define our fully discrete system as follows:
\begin{problem}[Full discretization] 
\label{prob:fullydiscreteweight} 
Find $u_{h,\tau} \in Q_T^{h,\tau}$ and $\phi_{h,\tau} \in B_T^{h,\tau}$ such that 
\begin{align}
 \label{eq:vp1htauweight}
\product{\widehat{\dt u}_{h,\tau}^n}{v_h}_\Omega 
+ \product{\widehat{\nabla u}_{h,\tau}^n}{\nabla v_h}_\Omega 
- \product{\widehat{\phi}_{h,\tau}^n}{v_h}_{\Gamma} &= \dual{\widehat{f}^n}{v_h}_\Omega,\\
\label{eq:vp2htauweight}
\product{(1/2-\K) \widehat{u}_{h,\tau}^n} {\psi_h}_{\Gamma} 
+ \product{\V\widehat{\phi}_{h,\tau}^n}{\psi_h}_{\Gamma} &= \product{\widehat{g}^n}{\psi_h}_{\Gamma} 
\end{align}
for all $v_h \in H^h \subset H^1(\Omega)$ and $\psi_h \in B^h \subset H^{-1/2}(\Gamma)$
 and for all $1 \le n \le N$.
\end{problem}
\begin{remark}
 \label{rem:classicalEuler}
 We have chosen the piecewise linear weight function $\omega^n(t)$ in~\cref{eq:hatv} such that
 for all $n\in\mathbb{N}$, $u_{h,\tau} \in Q_T^{h,\tau}$, and $\phi_{h,\tau} \in B_T^{h,\tau}$ 
 there holds that
 \begin{align}
  \label{eq:identities}
 \widehat{u}_{h,\tau}^n = u_{h,\tau}^n, \qquad 
 \widehat{\dt u}_{h,\tau}^n =  \dtau u_{h,\tau}^n= \frac{1}{\tau^n} (u_{h,\tau}^n - u_{h,\tau}^{n-1}), 
 \quad \text{ and} \quad
 \widehat{\phi}_{h,\tau}^n = \phi_{h,\tau}^n.
 \end{align}
 Thus the discrete system \cref{prob:fullydiscreteweight} is equivalent to 
 \begin{align}
  \label{eq:vp1htau}
 \product{\dtau u_{h,\tau}^n}{v_h}_\Omega 
 + \product{\nabla u_{h,\tau}^n}{\nabla v_h}_\Omega 
 - \product{\phi_{h,\tau}^n}{v_h}_{\Gamma} &= \dual{\widehat{f}^n}{v_h}_\Omega,\\
 \label{eq:vp2htau}
 \product{(1/2-\K) u_{h,\tau}^n} {\psi_h}_{\Gamma} 
 + \product{\V\phi_{h,\tau}^n}{\psi_h}_{\Gamma} &= \product{\widehat{g}^n}{\psi_h}_{\Gamma} 
 \end{align}
 for all $v_h \in H^h \subset H^1(\Omega)$ and $\psi_h \in B^h \subset H^{-1/2}(\Gamma)$,
  and for all $1 \le n \le N$.
 Hence, the fully discrete scheme \cref{prob:fullydiscreteweight}
 amounts to a discretization of \cref{prob:semidiscrete} 
 in time by a variant of the implicit Euler method, 
 i.e., it differs only in the right-hand side which is 
 treated in a special way in 
 order to reduce the regularity requirements on the data.
  An error analysis of the coupling with 
  the classical implicit Euler scheme and other
  time discretizations in the natural energy norm is also possible. However, one needs the usual
  Taylor expansions and therefore some regularity on the data $\tilde f$, $\tilde g$, $\tilde h$, 
  and the solution. 
\end{remark}
\begin{remark}
 \label{rem:consistency}
 By testing~\cref{eq:vp1}--\cref{eq:vp2} with $v=v_h$ and $\psi=\psi_h$, 
 multiplication with the weight function $\omega^n$, and 
 integration over the time interval $[t^{n-1},t^n]$, one can see that  
 \begin{align*}
 \dual{\widehat{\dt u}^n\!}{v_h}_\Omega + \product{\widehat{\nabla u}^n\!}{\nabla v_h}_\Omega 
 - \dual{\widehat{\phi}^n}{v_h}_{\Gamma} &= \dual{\widehat{f}^n}{v_h}_\Omega,\\
 \dual{(1/2-\K) \widehat{u}^n}{\psi_h}_{\Gamma} + \dual{\V  \widehat{\phi}^n}{\psi_h}_{\Gamma} 
 &= \dual{\widehat{g}^n}{\psi_h}_{\Gamma}
 \end{align*}
 for all $v_{h} \in H^h$, $\psi_h \in B^h$, and all $1 \le n \le N$. 
 This shows that the fully discrete scheme~\cref{eq:vp1htauweight}--\cref{eq:vp2htauweight} 
 is a Petrov-Galerkin approximation and thus is
 consistent with the variational problem~\cref{eq:vp1}--\cref{eq:vp2}.
\end{remark}
In the following, we derive error estimates for the fully discrete scheme in the energy norm
by an extension of our arguments for the analysis of the Galerkin semi-discretization.
Let us start with establishing the corresponding fully discrete energy estimate.
\begin{lemma}[Well-posedness] 
\label{lem:wellposedhtau}
Let~\cref{as:A1} hold and $\tau \le 1/4$. 
Then for any $f\in H'_T$ and $g\in B'_T$, \cref{prob:fullydiscreteweight} admits a unique solution and 
\begin{align}
 \label{eq:energyht}
\norm{u_{h,\tau}}{H_T} + \norm{\phi_{h,\tau}}{B_T} 
\le C e^{2N\tau} \big( \norm{f}{H_T'} + \norm{g}{B_T'}\big)
\end{align}
with a constant $C>0$ that depends only on the domain $\Omega$.
If the bilinear form $a(v,\psi;v,\psi)$ is elliptic 
there is no constant factor $e^{2N\tau}$ on the right-hand side of~\cref{eq:energyht}.
\end{lemma}
\begin{proof}
We recall the notation of \cref{lem:wellposedh} with $\phi^n_{h,\tau}=\S_h u^n_{h,\tau}+\R_h \widehat g^n$
and $\tilde a_h(u^n_{h,\tau},v_h) = \product{\nabla u^n_{h,\tau}}{\nabla v_h}_\Omega 
- \product{\S_h u^n_{h,\tau}}{v_h}_{\Gamma}$,
where $\S_h$ and $\R_h$ are defined in~\cref{eq:Sh} and~\cref{eq:Rh}, respectively.
Next we rewrite the equivalent formulation \cref{eq:vp1htau}--\cref{eq:vp2htau} 
of our discrete \cref{prob:fullydiscreteweight} as
\begin{align*}
\frac{1}{\tau^{n}} \product{u_{h,\tau}^{n} - u_{h,\tau}^{n-1}}{v_h}_\Omega 
+ \tilde a_h(u_{h,\tau}^{n},v_h) = \dual{\widehat{f}^{n}}{v_h}_{\Omega} 
+ \product{\R_h\widehat{g}^{n}}{v_h}_{\Gamma}.
\end{align*}
By testing with $v_h = u_{h,\tau}^{n}$ 
and using the relation $-ab=-\frac{1}{2}a^2-\frac{1}{2} b^2+\frac{1}{2} (a-b)^2$,
we apply the Cauchy-Schwarz, trace, and Young inequalities as well as the 
G\r arding inequality~\cref{eq:ahgarding} for the bilinear form $\tilde a_h(\cdot,\cdot)$ 
to get
\begin{align*}
&\frac{1}{2\tau^n} \norm{u_{h,\tau}^n}{L^2(\Omega)}^2 
-\frac{1}{2\tau^n} \norm{u_{h,\tau}^{n-1}}{L^2(\Omega)}^2 
+ \frac{1}{2\tau^n} \norm{u_{h,\tau}^n - u_{h,\tau}^{n-1}}{L^2(\Omega)}^2 
+ \alpha \norm{u_{h,\tau}^{n}}{H^1(\Omega)}^2 \\
&\qquad\le 
 \norm{u_{h,\tau}^{n}}{L^2(\Omega)}^2 
+ \frac{\alpha}{2}\norm{u_{h,\tau}^{n}}{H^1(\Omega)}^2 
+ \frac{1}{\alpha} \norm{\widehat{f}^{n}}{H^1(\Omega)'}^2 
+ \frac{C_\V^{-1}C_{tr}^2}{\alpha} \norm{\widehat{g}^{n}}{H^{1/2}(\Gamma)}^2.
\end{align*}
%
Additionally, we have used $\norm{\R_h \widehat{g}^{n}}{H^{-1/2}(\Gamma)} \le 
C_\V^{-1} \norm{\widehat{g}^{n}}{H^{1/2}(\Gamma)}$
for the operator $\R_h$ defined in~\cref{eq:Rh}, 
where $C_\V$ is the ellipticity constant of $\V$.
This shows that the problems are uniquely solvable at every time step. 
Multiplying with $2\tau^n(1-2\tau)^{n-1}$, rearranging the terms, 
and using the fact that $\tau^n\leq \tau\leq 1/4$,
a Gronwall argument, see, e.g.,~\cite{Wheeler:1973}, 
leads to
\begin{align}
 \label{eq:energyht1}
\norm{u_{h,\tau}^N}{L^2(\Omega)}^2 + \alpha\sum_{n=1}^N \tau^n \norm{u_{h,\tau}^{n}}{H^1(\Omega)}^2 
\le C e^{2N\tau} \sum_{n=1}^N \tau^n (\norm{\widehat{f}^{n}}{H^1(\Omega)'}^2 
+\norm{\widehat{g}^{n}}{H^{1/2}(\Gamma)}^2),
\end{align}
with a constant $C>0$.
Since $u_{h,\tau}$ and $\phi_{h,\tau}$ are piecewise linear and constant, respectively,
we easily see that
\begin{align}
 \label{eq:energyht2}
 \norm{u_{h,\tau}}{H_T}^2\leq \frac{4}{3}\sum_{n=1}^N \tau^n \norm{u_{h,\tau}^{n}}{H^1(\Omega)}^2
 \quad\text{ and }\quad
  \norm{\phi_{h,\tau}}{B_T}^2\leq \sum_{n=1}^N \tau^n \norm{\phi_{h,\tau}^{n}}{H^{-1/2}(\Omega)}^2.
\end{align}
For the right-hand side of~\cref{eq:energyht1} it follows 
directly by the Cauchy-Schwarz inequality and 
$\norm{\omega^n(t)}{L^2(t^{n-1},t^n)}^2=4\tau^n$ that
\begin{align}
 \label{eq:energyht3}
\sum_{n=1}^N \tau^n \norm{\widehat{f}^{n}}{H^1(\Omega)'}^2
\leq 4\norm{f}{H_T'}^2\qquad\text{and}\qquad
\sum_{n=1}^N \tau^n \norm{\widehat{g}^{n}}{H^{1/2}(\Gamma)}^2)
\leq 4\norm{g}{B_T'}^2   
\end{align} 
With~\cref{eq:vp2htau} and the same arguments as for~\cref{eq:energyphih} we get 
the bound
\begin{align}
 \label{eq:energyht4}
 C_\V\norm{\phi^n_{h,\tau}}{H^{-1/2}(\Gamma)}
  \le (1/2+C_\K) C_{tr} \norm{u^n_{h,\tau}}{H^1(\Omega)} 
 + \norm{\widehat{g}^n}{H^{1/2}(\Gamma)}.
\end{align}
Now the energy estimate~\cref{eq:energyht} follows from~\cref{eq:energyht1}--\cref{eq:energyht4}.
\end{proof}
With similar arguments as used for the analysis on the semi-discrete level, 
we also obtain a bound for the time derivatives $\dt u_{h,\tau}$ of the discrete solution.
\begin{lemma}[Energy estimate] 
\label{lem:fullydiscreteenergy}
Let~\cref{as:A1}--\cref{as:A2} hold and $\tau \le 1/4$. Then 
\begin{align}
\label{eq:fullydiscreteenergy}
\norm{u_{h,\tau}}{Q_T} + \norm{\phi_{h,\tau}}{B_T}
\le C \big( \norm{f}{H_T'} +\norm{g}{B_T'}\big). 
\end{align}
The constant $C>0$ depends only on the domain $\Omega$ and the time horizon $T$.
\end{lemma}
\begin{proof}
In view of \cref{lem:wellposedhtau}, we only have to estimate 
\begin{align*}
 \norm{\dt u_{h,\tau}}{H'_T}^2=
 \sum_{n=1}^N \tau^n\norm{\dtau u_{h,\tau}^n}{H^1(\Omega)'}^2.
\end{align*}
With similar reasoning as in \cref{lem:energyh}, we obtain
\begin{align}
 \label{eq:fullydiscreteenergy1}
\norm{\dtau u_{h,\tau}^n}{H^1(\Omega)'} 
&= \sup_{0\not= v \in H^1(\Omega)} 
\frac{\product{\dtau u_{h,\tau}^n}{v}_\Omega}{\norm{v}{H^1(\Omega)}}
 = \sup_{0\not= v \in H^1(\Omega)} 
 \frac{\product{\dtau u_{h,\tau}^n}{P_h v}_\Omega}{\norm{v}{H^1(\Omega)}}.
\end{align}
By equation~\cref{eq:vp1htau} and the Cauchy-Schwarz inequality, we further get
\begin{align*}
\product{\dtau u_{h,\tau}^n}{P_h v}_\Omega \le \big(\norm{u_{h,\tau}^n}{H^1(\Omega)} 
+ C_{tr} \norm{\phi_{h,\tau}^n}{H^{-1/2}(\Gamma)} +\norm{\widehat{f}^{n}}{H^{1}(\Omega)'}\big) 
\norm{P_h v}{H^1(\Omega)}.
\end{align*}
The $H^1$-stability assumption~\cref{as:A2} therefore yields for~\cref{eq:fullydiscreteenergy1} 
\begin{align*}
\norm{\dtau u_{h,\tau}^n}{H^1(\Omega)'} 
 \le C \big( \norm{u_{h,\tau}^n}{H^1(\Omega)} 
 + \norm{\phi_{h,\tau}^n}{H^{-1/2}(\Gamma)} + \norm{\widehat{f}^{n}}{H^1(\Omega)'} \big).
\end{align*}
The assertion now follows by squaring this estimate, multiplying with $\tau^{n}$, 
summation over $n$, and the estimates~\cref{eq:energyht1},~\cref{eq:energyht3},~\cref{eq:energyht4},
and~\cref{eq:energyht}.
\end{proof}
Now we prove the main result of this work.
\begin{theorem}[Quasi optimality of the fully discrete scheme] 
\label{thm:main}
Let~\cref{as:A1}--\cref{as:A2} hold and $\tau \le 1/4$. Furthermore, $(u,\phi) \in Q_T \times B_T$ 
and $(u_{h,\tau},\phi_{h,\tau}) \in Q_T^{h,\tau} \times B_T^{h,\tau}$ denote the solutions of
\cref{prob:variational} and \cref{prob:fullydiscreteweight}, respectively. Then 
\begin{align}
 \label{eq:fullquasioptimal}
\norm{u - u_{h,\tau}}{Q_T} + \norm{\phi - \phi_{h,\tau}}{B_T}  
\le C \big( \norm{u - \tilde u_{h,\tau}}{Q_T} + \norm{\phi - \tilde \phi_{h,\tau}}{B_T}\big)
\end{align}
for all functions $\tilde u_{h,\tau} \in Q_T^{h,\tau}$ and $\tilde \phi_{h,\tau} \in B_T^{h,\tau}$.  
The constant $C>0$ in this estimate depends only on the domain $\Omega$ and the time horizon $T$.
\end{theorem}
\begin{proof}
The result follows with similar arguments as used in the proof of
\cref{thm:quasioptimality}.
Let $\tilde u_{h,\tau} \in Q_T^{h,\tau}$ and $\tilde \phi_{h,\tau} \in B_T^{h,\tau}$ be arbitrary.
Then we split the error
\begin{align*}
\norm{u - u_{h,\tau}}{Q_T}\leq
\norm{u-\tilde u_{h,\tau}}{Q_T}+\norm{\tilde u_{h,\tau} - u_{h,\tau}}{Q_T}, \\
\norm{\phi - \phi_{h,\tau}}{B_T}\leq
\norm{\phi - \tilde\phi_{h,\tau}}{B_T}+\norm{\tilde \phi_{h,\tau} - \phi_{h,\tau}}{B_T}.
\end{align*}
To estimate the \emph{discrete error}
we recall the consistency of the
fully discrete scheme~\cref{eq:vp1htauweight}--\cref{eq:vp2htauweight} 
with the variational problem~\cref{eq:vp1}--\cref{eq:vp2}, see \cref{rem:consistency}. 
Hence, the discrete error components
$w_{h,\tau}:=\tilde u_{h,\tau} - u_{h,\tau}$ 
$\rho_{h,\tau}:=\tilde \phi_{h,\tau} - \phi_{h,\tau}$ fulfill the system
\begin{align}
 \label{eq:fullquasi1}
 \dual{\widehat{\dt w}^n_{h,\tau}\!}{v_h}_\Omega + \product{\widehat{\nabla w}^n_{h,\tau}\!}{\nabla v_h}_\Omega 
 - \dual{\widehat{\rho}^n_{h,\tau}}{v_h}_{\Gamma} &= \dual{\widehat{F}^n}{v_h}_\Omega,\\
 \label{eq:fullquasi2}
 \dual{(1/2-\K) \widehat{w}^n_{h,\tau}}{\psi_h}_{\Gamma} + 
 \dual{\V  \widehat{\rho}^n_{h,\tau}}{\psi_h}_{\Gamma} 
 &= \dual{\widehat{G}^n}{\psi_h}_{\Gamma}
\end{align}
for all $v_{h} \in H^h$, $\psi_h \in B^h$, and all $1 \le n \le N$ 
with the averaged right-hand sides $\widehat{F}$ and $\widehat{G}$ obtained from
\begin{align*}
\dual{F(t)}{v}_{\Omega} 
  &:= \dual{\dt \tilde u_{h,\tau}(t) - \dt u(t)}{v}_\Omega\\ 
  &\qquad+ \product{\nabla \tilde u_{h,\tau}(t) - 
    \nabla u(t)}{\nabla v}_\Omega 
  - \dual{\tilde \phi_{h,\tau}(t) - \phi(t)}{v}_{\Gamma},\\
\dual{G(t)}{\psi}_{\Gamma} 
  &:= \dual{(1/2-\K) (\dt \tilde u_{h,\tau}(t) - \dt u(t))}{\psi}_{\Gamma} 
  + \dual{\V (\tilde \phi_{h,\tau}(t) - \phi(t))}{\psi}_{\Gamma}.
\end{align*}
for all $v\in H$ and $\psi\in B$.
%
Note that the system~\cref{eq:fullquasi1}--\cref{eq:fullquasi2} 
has the same form as~\cref{eq:vp1htauweight}--\cref{eq:vp2htauweight}
with the right-hand sides $\widehat{F}^n$ 
and $\widehat{G}^n$.
Thus we can apply the energy estimate~\cref{eq:fullydiscreteenergy} of \cref{lem:fullydiscreteenergy}.
The estimates
\begin{align*}
\norm{F}{H'_T} 
&\le  C \big( \norm{u - \tilde u_{h,\tau}}{Q_T} + \norm{\phi - \tilde \phi_{h,\tau}}{B_T}\big),  \\
\norm{G}{B'_T} 
& \le C \big( \norm{u - \tilde u_{h,\tau}}{H_T} + \norm{\phi - \tilde \phi_{h,\tau}}{B_T}\big),
\end{align*}
and the error splitting complete the proof for~\cref{eq:fullquasioptimal}.
\end{proof}

\begin{remark} 
The time discretization strategy can also be applied directly to the continuous 
variational problem \cref{eq:vp1}--\cref{eq:vp2}. Let us denote by
\begin{align*}
Q_T^\tau &= \set{u \in Q_T}{u|_{[t^{n-1},t^n]} \text{ is linear in } t} \qquad \text{and} \\
B_T^\tau &= \set{\phi \in B_T}{\phi|_{(t^{n-1},t^n]} \text{ is constant in } t} 
\end{align*}
the corresponding function spaces and let $(u_\tau, \phi_\tau) \in Q_T^\tau \times B_T^\tau$ 
be the respective solutions obtained by time discretization of the continuous variational problem.
The well-posedness of this time-discretized problem follows by simply 
setting $Q_T^{h,\tau}=Q_T^\tau$ and $B_T^{h,\tau}=B_T^\tau$ in the above results. 
As a consequence, we also obtain the quasi-optimal error bound
\begin{align*}
\norm{u - u_\tau}{Q_T} + \norm{\phi - \phi_\tau}{Q_T} \le C \big( \norm{u - \tilde u_\tau}{Q_T} 
+ \norm{\phi - \tilde \phi_\tau}{B_T}\big). 
\end{align*}
for all $\tilde u_\tau \in Q_T^\tau$ and $\tilde \phi_\tau \in B_T^\tau$
with a constant $C$ being independent of $u$, $\phi$ and the temporal grid. 
The condition~\cref{as:A2} is not required for this result to hold true.
\end{remark}

\begin{remark}
Explicit error bounds for the time discretization of the continuous and the semi-discrete 
variational problem can also be obtained via the usual Taylor estimates 
under some regularity 
assumptions on the solution. As we will see in the next section, we obtain linear convergence 
with respect to $\tau$ and independent of the spatial approximation.
Furthermore, other time discretization schemes are possible here, e.g., choose
$w^n(t)=1$ in~\cref{eq:hatv}. Then the identities~\cref{eq:identities} are
 \begin{align*}
 \widehat{u}_{h,\tau}^n = (u_{h,\tau}^n+u_{h,\tau}^{n-1})/2, \qquad 
 \widehat{\dt u}_{h,\tau}^n =  \dtau u_{h,\tau}^n= \frac{1}{\tau^n} (u_{h,\tau}^n - u_{h,\tau}^{n-1}), 
 \quad \text{ and } 
 \widehat{\phi}_{h,\tau}^n = \phi_{h,\tau}^n.
 \end{align*}
 and the discrete system \cref{prob:fullydiscreteweight} becomes a variant of the Crank-Nicolson time
 discretization.
\end{remark}

\section{Error estimates for a FEM-BEM discretization} 
\label{sec:fembem}

In this section we discuss a space discretization with finite and boundary elements. 
Together with the 
time discretization of the previous section, this yields to a fully discrete method which converges 
uniformly and exhibits order optimal convergence rates under minimal 
regularity assumptions on the solution.
We assume in the following that 
\begin{align}
 \label{as:A3}
 \tag{A3}
 &\T=\{T\} \text{ is a conforming triangulation of the domain }\Omega; \text{ see~\cite{Ciarlet:1978-book}}. \\
 \label{as:A4}
 \tag{A4}
 &\E_{\Gamma} = \{ E \} \text{ is a segmentation of the boundary }\Gamma \text{ into straight edges}. 
\end{align}
Note that condition~\cref{as:A3} and~\cref{as:A4} particularly imply that $\Gamma$ is a polygon
and that the surface mesh $\E_{\Gamma}$ is in general decoupled from the mesh $\T|_\Gamma$ of the domain.
\begin{remark}
An analysis for curved boundaries can be found in~\cite{ErathSchorr:2017-2}.
In~\cite{Gonzalez:2006}, 
curved finite elements are considered for the symmetric FEM-BEM coupling in two dimensions 
for a time-dependent problem.  
\end{remark}
As usual we denote by $\rho_T$ and $h_T$ the inner circle radius and diameter of the 
triangle $T\in\T$ and by $h_E$ the length of the edge $E\in\E_{\Gamma}$. We further set 
$h = \max\{\max_{T} h_T,\max_{E} h_E\}$ and assume that 
\begin{align}
 \label{as:A5}
 \tag{A5}
 \begin{split}
 &\text{the partition }(\T,\E_{\Gamma}) \text{ is }\eta\text{-quasi-uniform with }\eta>0, \text{ i.e. },\\
&\eta h \le \rho_T \le h_T \le h 
\qquad \text{and} \qquad 
\eta h \le h_E \le h
\qquad \text{for all } T \in \T , \ E \in \E_{\Gamma}.
\end{split}
\end{align}
For the Galerkin semi-discretization in space we utilize the standard approximations
\begin{align}
 \label{eq:spaceHh}
H^h &= \set{v \in C(\Omega)}{v|_T \in \P^1(T) \text{ for all } T \in \T }\qquad \text{and} \\
\label{eq:spaceBh}
B^h &= \set{\psi \in L^2(\Gamma)}{\psi|_E \in \P^0(E) \text{ for all } E \in \E_{\Gamma}} 
\end{align}
consisting of globally continuous and
piecewise linear functions over $\T$ and piecewise constant functions over $\E_{\Gamma}$, 
respectively. We denote by 
$P_h: L^2(\Omega) \to H^h$ and $\Pi_h : H^{-1/2}(\Gamma) \to B^h$ the 
$L^2(\Omega)$- and the $H^{-1/2}(\Gamma)$-orthogonal projection, respectively. 
\begin{lemma} 
\label{lem:approxerror}
Let~\cref{as:A1} and~\cref{as:A3}--\cref{as:A5} hold. Then~\cref{as:A2} is valid with a constant $C_P$ 
independent of the mesh-size. Moreover, the operator $P_h$ can be extended to a bounded linear 
operator on $H^1(\Omega)'$. Hence, for all $0 \le s \le 1$ and $0 \le s_e \le 3/2$ we have
\begin{align*}
\norm{u - P_h u}{H^{1}(\Omega)} &\le C h^s \norm{u}{H^{1+s}(\Omega)},\qquad u\in H^{1+s}(\Omega), \\
\norm{u - P_h u}{H^{1}(\Omega)'} &\le C h^s \norm{u}{H^{1-s}(\Omega)'},\qquad u\in H^{1-s}(\Omega)',  \\
\norm{\phi - \Pi_h \phi}{H^{-1/2}(\Gamma)} &\le C h^{s_e} 
\norm{\phi}{H^{-1/2+s_e}(\Gamma)},\qquad \phi\in H^{-1/2+s_e}(\Gamma).
\end{align*}
The constant $C>0$ is independent of the particular choice of the triangulation. 
\end{lemma}
\begin{proof}
The assertion about $\phi$ follows from \cite[Th. 10.4]{Steinbach:2008-book}.
Validity of condition~\cref{as:A2} for these particular function spaces 
has been shown in~\cite{Costabel:1990} via an inverse inequality. 
Now we turn to the remaining estimates:
Let $P^1_h : H^1(\Omega) \to H^h$ be the $H^1$-orthogonal projection defined by 
\begin{align*}
\product{P^1_h u}{v_h}_{H^1(\Omega)} = \product{u}{v_h}_{H^1(\Omega)} \qquad \text{for all } v_h \in H^h,
\end{align*}
and recall that $\norm{u-P^1_h u}{H^1(\Omega)} \le C' h^s \norm{u}{H^{1+s}(\Omega)}$ 
for $0 \le s \le 1$; see, e.g.,~\cite{Brenner:2008-book}.
Then 
\begin{align*}
\norm{u - P_h u }{H^1(\Omega)} 
&\le \norm{u - P_h P_h^1 u}{H^1(\Omega)} + \norm{P_h (u - P_h^1 u)}{H^1(\Omega)} \\
&\le (1 + C_P) \norm{u - P^1_h u}{H^1(\Omega)} \le (1+C_P) C' h^s \norm{u}{H^{1+s}(\Omega)},
\end{align*}
where we used the projection property of $P_h$, condition~\cref{as:A2},
and the approximation properties of $P^1_h$ in the last two steps.
By definition of the dual norm, we further have 
\begin{align*}
\norm{u - P_h u}{H^{1}(\Omega)'} 
&= \sup_{0\not= v \in H^1(\Omega)} \frac{\product{u - P_h u}{v}_{\Omega}}{\norm{v}{H^1(\Omega)}} \\
&= \sup_{0\not= v \in H^1(\Omega)} \frac{\product{u}{v - P_h v}_{\Omega}}{\norm{v}{H^1(\Omega)}}
\le C h \norm{u}{L^2(\Omega)}.
\end{align*}
Here we used the standard estimate $\norm{v - P_h v}{L^2(\Omega)} \le C h \norm{v}{H^1(\Omega}$ 
for the $L^2$-projection in the last step. 
With a similar duality argument and condition~\cref{as:A2}, one can further see that 
$\norm{P_h u}{H^1(\Omega)'} \le C_P \norm{u}{H^1(\Omega)'}$ for all functions in $L^2(\Omega)$. 
By density of $L^2(\Omega)$ in $H^1(\Omega)'$, 
we can extend $P_h$ to a bounded linear operator on $H^1(\Omega)'$, and obtain 
\begin{align*}
\norm{u - P_h u}{H^1(\Omega)'} \le (1+C_P) \norm{u}{H^1(\Omega)'}. 
\end{align*}
Noting that $L^2(\Omega) = H^0(\Omega) = H^0(\Omega)'$ and interpolating the two latter bounds 
now allows us to establish the second estimate for $u$ which completes the proof.
\end{proof}

\begin{remark} 
Due to the results of~\cite{Bramble:2001} and~\cite{Bank:2014}, 
the assertions of \cref{lem:approxerror} also holds true on rather general shape-regular meshes 
under a mild growth condition on the local mesh size. With standard arguments, these estimates 
can also be generalized to polynomial approximations of higher order. All results that 
are presented below thus can be extended to such more general situations.
\end{remark}
As a consequence of these approximation error bounds and the quasi-best approximation of 
the semi-discretization, we obtain the following quantitative error estimates.
\begin{theorem} 
Let~\cref{as:A1}--\cref{as:A5} hold and denote by $(u,\phi)$ and $(u_h,\phi_h)$ 
the solutions of \cref{prob:variational} and 
\cref{prob:semidiscrete}, respectively. Then 
\begin{align*}
&\norm{u - u_h}{Q_T} + \norm{\phi - \phi_h}{B_T} \\
&\quad\le C h^s \big( \norm{u}{L^2(0,T;H^{1+s}(\Omega))} + \norm{\dt u}{L^2(0,T;H^{1-s}(\Omega)')} 
+ \norm{\phi}{L^2(0,T;H^{s-1/2}(\Gamma))}\big) 
\end{align*}
for all $0 \le s \le 1$, $u(t)\in H^{1+s}(\Omega)$, $\dt u \in H^{1-s}(\Omega)'$, 
$\phi(t)\in H^{-1/2+s}(\Gamma)$, and for a.e. $t\in [0,T]$.
\end{theorem}
\begin{proof}
The result follows directly from \cref{thm:quasioptimality} and \cref{lem:approxerror}. 
\end{proof}

\begin{remark} 
Let us emphasize that the estimate of the theorem is optimal with respect to both, 
the approximation properties of the spaces $Q_T^h$ and $B_T^h$ and the smoothness 
requirements on the solution. Furthermore, the method even converges without any 
smoothness assumptions on the 
solution, i.e., for all $u \in Q_T$ and $\phi \in B_T$.
\end{remark}

For the full discretization we will also need the $L^2$-projection in time, i.e., operators 
$P^\tau : L^2(0,T;L^2(\Omega)) \to Q_T^\tau$ and
$\Pi^\tau : L^2(0,T;H^{-1/2}(\Gamma)) \to B_T^\tau$. These satisfy
\begin{align*}
\norm{v - P^\tau v}{Q_T} &\le C \tau^r \big(\norm{\dt v}{H^r(0,T;H^1(\Omega)')}
+ \norm{v}{H^r(0,T;H^1(\Omega))}\big),\qquad 0\leq r\leq 1, \\
\norm{\psi - \Pi^\tau \psi}{B_T} &\le C \tau^r \norm{\psi}{H^r(0,T;H^{-1/2}(\Gamma))},\qquad 0\leq r\leq 1.
\end{align*}
Then we obtain the following result for the fully discrete scheme. 
\begin{theorem}
\label{th:errorestimateorder}
Let~\cref{as:A1}--\cref{as:A5} hold and $\tau \le 1/4$. Further we denote by 
$(u,\phi)$ and $(u_{h,\tau},\phi_{h,\tau})$ the solutions of 
\cref{prob:variational} and \cref{prob:fullydiscreteweight}, respectively. Then 
\begin{align*}
&\norm{u - u_{h,\tau}}{Q_T} + \norm{\phi - \phi_{h,\tau}}{B_T} \\
&\quad\le C_1 h^s \big( \norm{u}{L^2(0,T;H^{1+s}(\Omega))} + \norm{\dt u}{L^2(0,T;H^{1-s}(\Omega)')} 
+ \norm{\phi}{L^2(0,T;H^{s-1/2}(\Gamma))} \big) \\
& \qquad \qquad + C_2 \tau^r \big( \norm{\dt u}{H^r(0,T;H^1(\Omega)')} + \norm{u}{H^r(0,T;H^1(\Omega))} 
+ \norm{\phi}{H^r(0,T;H^{-1/2}(\Gamma))} \big) 
\end{align*}
for all $0 \le s \le 1$ and $0 \le r \le 1$ with
$u\in H^r(0,T;H^{1+s}(\Omega))$, $\dt u\in H^r(0,T;H^{1-s}(\Omega)')$, and $\phi\in H^r(0,T;H^{-1/2+s}(\Gamma))$.
The constants $C_1,C_2>0$ depend only on the domain $\Omega$ and the time horizon $T$.
\end{theorem}
\begin{proof}
By the triangle inequality, we obtain
\begin{align*}
\norm{u - P^\tau P_h u}{Q_T} &\le \norm{u - P_h u}{Q_T} 
+ \norm{P_h u - P^\tau P_h u}{Q_T},\\
\norm{\phi - \Pi^\tau \Pi_h\phi}{B_T} &\le \norm{\phi - \Pi_h\phi}{B_T} 
+ \norm{\Pi_h\phi - \Pi^\tau\Pi_h\phi}{B_T}.
\end{align*}
The first term in each line can be estimated by \cref{lem:approxerror}. 
Since the projection operators commute, we can change their order in the second term in each line. 
Then we use the stability of the spatial projection operators guaranteed 
by \cref{lem:approxerror} and the approximation properties of the 
time projections $P^\tau$. We obtain 
\begin{align*}
\norm{P_h u - P^\tau P_h u}{Q_T} 
&\le C \tau^r \big(\norm{\dt u}{H^r(0,T;H^1(\Omega)')} + \norm{u}{H^r(0,T;H^1(\Omega))}\big), \\
\norm{\Pi_h\phi - \Pi^\tau \Pi_h\phi}{B_T} 
&\le C \tau^r \norm{\phi}{H^r(0,T;H^{-1/2}(\Gamma))}.
\end{align*}
Now we apply \cref{thm:main}
with $\tilde u_{h,\tau}=P^\tau P_h u$ and $\tilde \phi_{h,\tau}=\Pi^\tau \Pi_h \phi$. 
The estimates from \cref{lem:approxerror} for the approximation errors yield to the assertion.
\end{proof}

\begin{remark}
From the previous result, we also obtain a corresponding estimate
\begin{align*}
\norm{u &- u_\tau}{Q_T} + \norm{\phi - \phi_\tau}{B_T} \\
&\le C \tau^r \big( \norm{\dt u}{H^r(0,T;H^1(\Omega)')} 
+ \norm{u}{H^r(0,T;H^1(\Omega))} + \norm{\phi}{H^r(0,T;H^{-1/2}(\Gamma))} \big) 
\end{align*}
for the approximation $(u_\tau,\phi_\tau)$ obtained by the time discretization scheme 
without additional Galerkin approximation in space. 
The proof of this result simply follows 
by setting $Q_T^h=Q_T$, $B_T^h=B_T$ and $Q_T^{h,\tau}=Q_T^\tau$, $B_T^{h,\tau}=B_T^\tau$ in the previous theorem. 
Note that the conditions~\cref{as:A2}--\cref{as:A5} are not required for this result to hold true.
\end{remark}

\begin{figure}[tbhp]
\centering 
\subfigure[Mesh for \cref{subsec:bspanalytic}.]{\label{subfig:meshlshape}
\includegraphics[width=.38\textwidth]{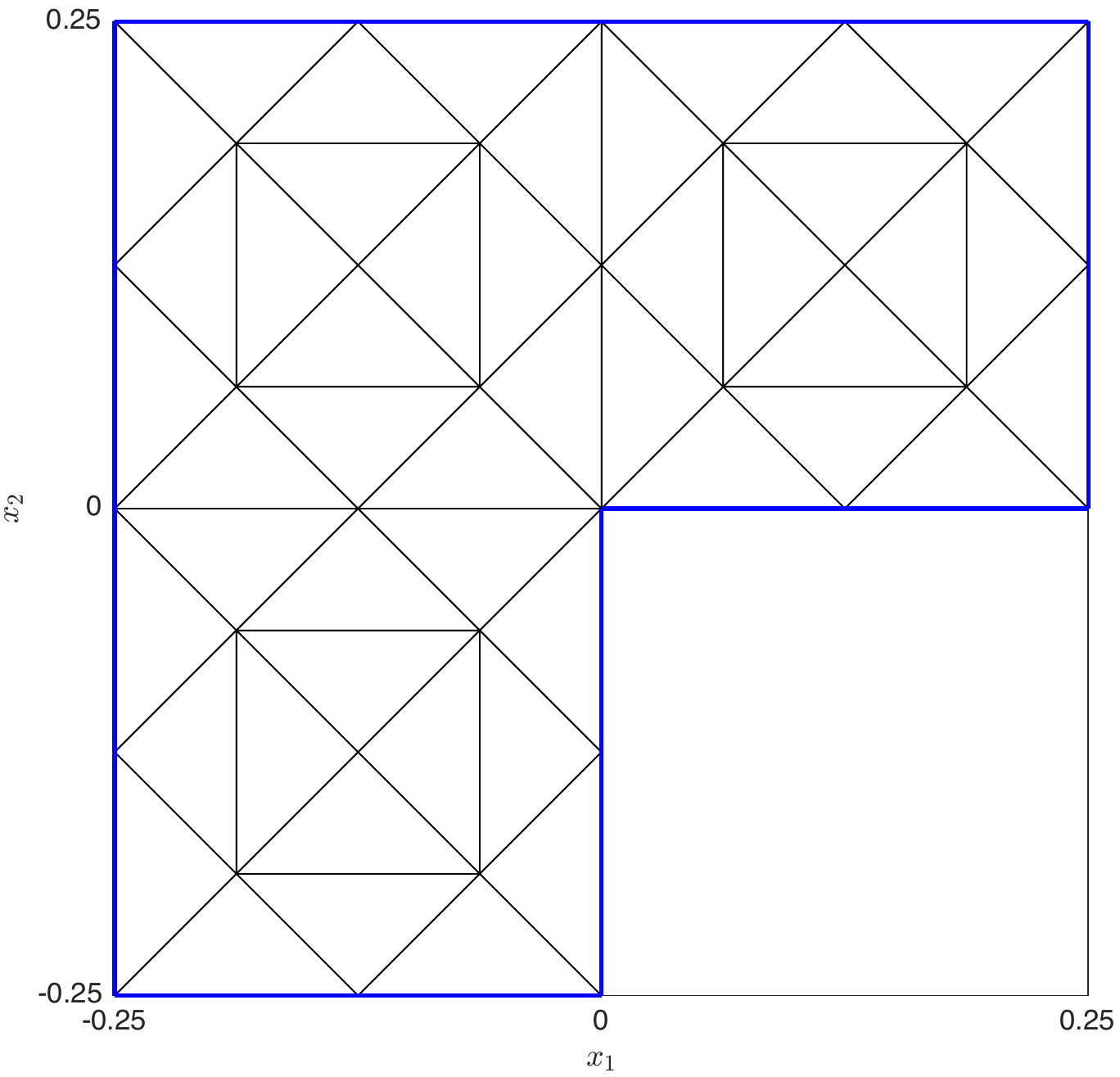}}
\hspace{0.075\textwidth}
\subfigure[Mesh for \cref{subsec:bspcap}.]{\label{subfig:meshcap}
\includegraphics[width=.37\textwidth]{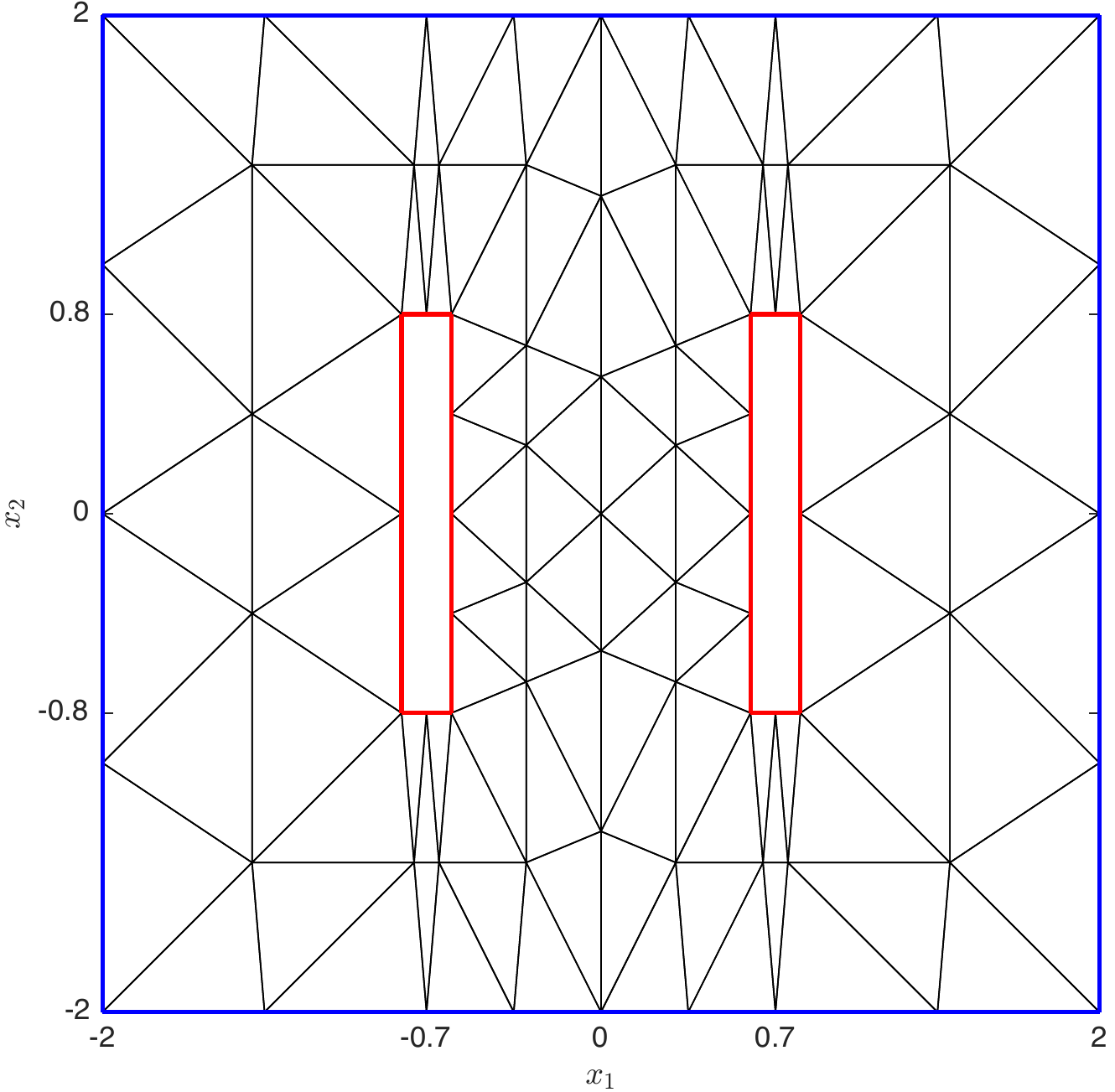}}
\caption{The initial triangle meshes for the examples.
The bold lines are the coupling boundary (blue) and the Dirichlet boundary (red).}
\label{fig:meshes}
\end{figure}
\section{Numerical illustration} 
\label{sec:numerics}
In this section we illustrate our theoretical findings by some numerical examples in $\mathbb{R}^2$ 
with the function spaces $H^h$ and $B^h$ defined in~\cref{eq:spaceHh} and~\cref{eq:spaceBh}, respectively.
For the implementation we use the equivalent system~\cref{eq:vp1htau}--\cref{eq:vp2htau}
instead of \cref{prob:fullydiscreteweight}, see \cref{rem:classicalEuler}.
The right-hand side is built from the model data $\tilde f$, $\tilde g$, 
$\tilde h$ with~\cref{eq:f} and \cref{eq:g}, and with the aid
of the weighted average operator~\cref{eq:hatv}. For these integrals we use Gauss quadrature
in space and time.
The calculations were performed using \textsc{Matlab} utilizing some functions 
from the \textsc{Hilbert}-package~\cite{HILBERT:2013-1} for assembling the matrices resulting 
from the integral operators $\V$ and $\K$. 
\begin{figure}
\centering
\begin{tikzpicture}
\begin{loglogaxis}[width=0.9\textwidth,
xlabel={$1/h$}, ylabel={\small error}, font={\scriptsize},
ymin=1e-6, 
legend style={font=\small, draw=none, fill=none, cells={anchor=west}, legend pos=south west}]
\addplot table [x=invmaxMeshsizeh,y=errorH1dual] {figures/mexicanhatfemsinus19112017.dat};
\addplot table [x=invmaxMeshsizeh,y=errorenergyV] {figures/mexicanhatfemsinus19112017.dat};
\addplot table [x=invmaxMeshsizeh,y=errorenergyVproj] {figures/mexicanhatfemsinus19112017.dat};
\addplot table [x=invmaxMeshsizeh,y=errorL2] {figures/mexicanhatfemsinus19112017.dat};
\addplot table [x=invmaxMeshsizeh,y=errorL2proj] {figures/mexicanhatfemsinus19112017.dat};
\addplot table [x=invmaxMeshsizeh,y=errorH1semi] {figures/mexicanhatfemsinus19112017.dat};
\addplot table [x=invmaxMeshsizeh,y=errorH1semiproj] {figures/mexicanhatfemsinus19112017.dat};
\addplot table [x=invmaxMeshsizeh,y=globalEnergy] {figures/mexicanhatfemsinus19112017.dat};
\addplot table [x=invmaxMeshsizeh,y=globalEnergyproj] {figures/mexicanhatfemsinus19112017.dat};

\logLogSlopeTriangle{0.925}{0.2}{0.49}{-1}{black}{\scriptsize};
\logLogSlopeTriangle{0.925}{0.2}{0.66}{-1}{black}{\scriptsize};

\legend{
$\norm{z_h^a}{H_T}$,
$\norm{\phi - \phi_{h,\tau}}{L^2(0,T;\V)}$,
$\norm{\overline{\phi}_h - \phi_{h,\tau}}{L^2(0,T;\V)}$,
$\norm{u-u_{h,\tau}}{L^2(0,T;L^2(\Omega))}$,
$\norm{\overline{u}_h-u_{h,\tau}}{L^2(0,T;L^2(\Omega))}$,
$\norm{\nabla(u-u_{h,\tau})}{L^2(0,T;L^2(\Omega))}$,
$\norm{\nabla(\overline{u}_h-u_{h,\tau})}{L^2(0,T;L^2(\Omega))}$,
$(\norm{u-u_{h,\tau}}{H_T}^2+\norm{z_h^a}{H_T}^2)^{1/2}+\norm{\phi - \phi_{h,\tau}}{L^2(0,T;\V)}$,
$(\norm{\overline{u}_h-u_{h,\tau}}{H_T}^2+\norm{z_h^a}{H_T}^2)^{1/2}+\norm{\overline{\phi}_h - \phi_{h,\tau}}{L^2(0,T;\V)}$}
\end{loglogaxis}
\end{tikzpicture}
\caption{The different error components of the solutions   
$u_{h,\tau}$ and $\phi_{h,\tau}$
for uniform refinement in time and space 
for the smooth example in \cref{subsubsec:bsp1}.
The added energy error norms 
$(\norm{u-u_{h,\tau}}{H_T}^2+\norm{z_h^a}{H_T}^2)^{1/2}+\norm{\phi - \phi_{h,\tau}}{L^2(0,T;\V)}$
and
$(\norm{\overline{u}_h-u_{h,\tau}}{H_T}^2+\norm{z_h^a}{H_T}^2)^{1/2}+\norm{\overline{\phi}_h - \phi_{h,\tau}}{L^2(0,T;\V)}$
show the first order convergence as predicted in \cref{th:errorestimateorder}.
}
\label{fig:errorbsp1}
\end{figure}
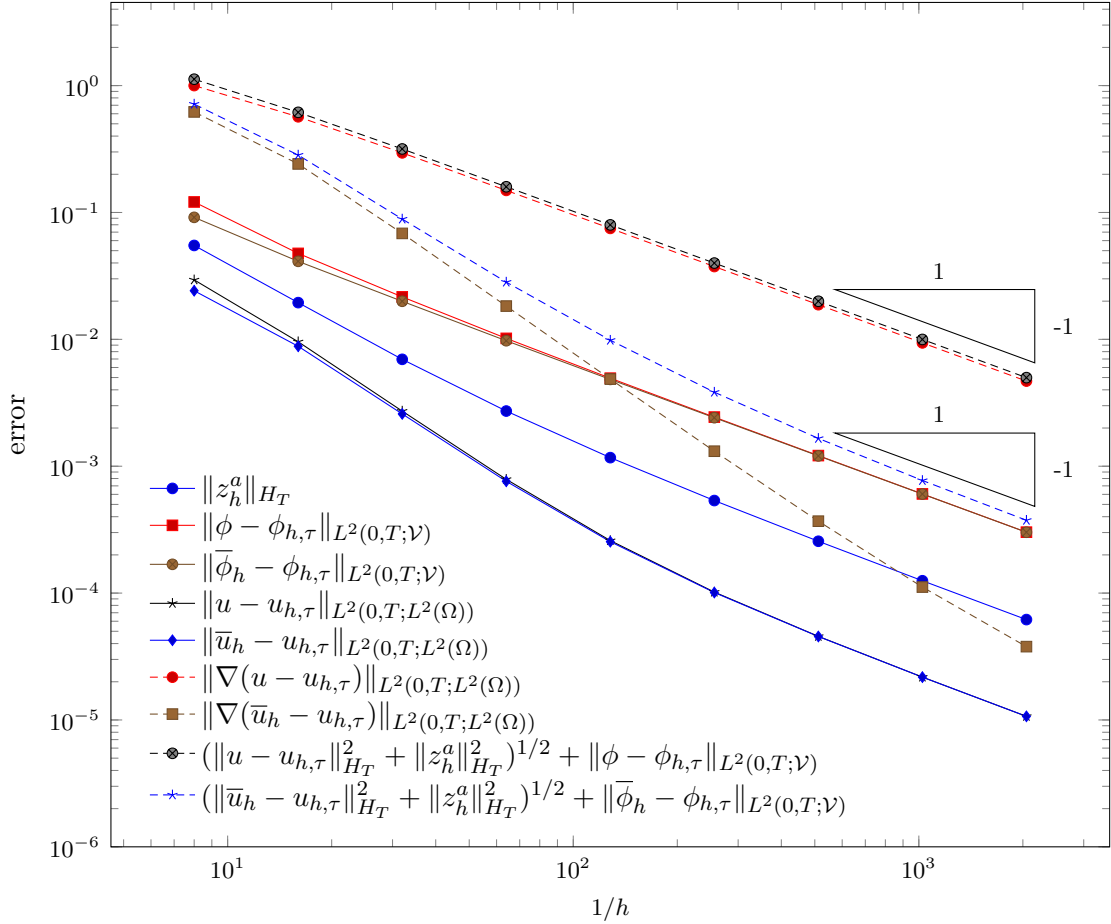
\subsection{Tests with analytical solutions}
\label{subsec:bspanalytic}
In the following, we discuss the convergence behaviour for 
three examples with analytical solutions. 
We consider the coupling problem~\cref{eq:model1}--\cref{eq:model6} 
on the classical L-shape $\Omega= (-1/4, 1/4)^2\setminus [0, 1/4] \times [-1/4, 0]$
and the time interval $[0,1]$. The uniform start triangulation (triangles) 
is plotted in \cref{subfig:meshlshape} with $h=0.125$. We use uniform time stepping, 
in particular, we begin with $\tau^n=\tau=0.05$. 
The refinement will be uniform for both, the space and the time grid, and
simultaneously. 
For all three
examples we prescribe the same analytical solution in the exterior domain $\Omega_e$, namely
\begin{align*}
u_e(x_1,x_2,t)&=(1-t)\log \sqrt{(x_1+0.125)^2+(x_2-0.125)^2}.
\end{align*}
Note that this solution is smooth in $\Omega_e$.
With the interior solutions given below we will calculate 
the right-hand side $\tilde f$ and the jumps $\tilde g$ and $\tilde h$
(from $u = u_e + \tilde g$ and 
$\partial_{\normal} u = \partial_{\normal} u_e + \tilde h$)
appropriately.
For the error discussion we also consider
the $L^2$-projected analytical solutions $\overline{u}_h(t)\in H^h$ of $u(t)$
and $\overline{\phi}_h(t)\in B^h$ of $\phi(t)$ for a fixed but arbitrary $t$. 
Note that the prescribed exterior solution
guarantees at least $\phi(t)\in L^2(\Gamma)$.
Hence, we may estimate the error as
\begin{align}
 \label{eq:l2splitting1}
 \norm{u-u_{h,\tau}}{Q_T}&\leq  \norm{u-\overline{u}_h}{Q_T}+\norm{\overline{u}_h-u_{h,\tau}}{Q_T},\\
 \label{eq:l2splitting2}
 \norm{\phi-\phi_{h,\tau}}{B_T}&\leq  \norm{\phi-\overline{\phi}_h}{B_T}
 +\norm{\overline{\phi}_h-\phi_{h,\tau}}{B_T}.
\end{align}
The convergence order of $\norm{u-\overline{u}_h}{Q_T}$ and
$\norm{\phi-\overline{\phi}_h}{B_T}$ are known a~priori. 
With the discrete error $e_h(t):=\overline{u}_h(t)-u_{h,\tau}(t)$ 
we can estimate the non computable dual norm 
$\norm{\dt e_h}{H'_T}^2=\int_0^T\norm{\dt e_h}{H'}^2$ in the following way.
Let $z_h^a\in H^h$ be the solution to the auxiliary problem
\begin{align*}
 \product{\nabla z_h^a}{\nabla v_h}_{\Omega} + \product{z_h^a}{v_h}_{\Omega} 
 =\product{\dt e_h}{v_h}_{\Omega},
\end{align*}
with $v_h=P_h v$ for all $v\in H$ and $P_h$ being the $L^2$-projection introduced
in \cref{sec:fembem}. 
Then the $H^1$-stability of $P_h$ and the definition of the auxiliary problem
lead to
\begin{align*}
 \norm{\dt e_h}{H^1(\Omega)'}
 &=\sup_{0\not =v\in H^1(\Omega)}\frac{\product{\dt e_h}{v}_{\Omega}}{\norm{v}{H^1(\Omega)}} \\
&=\sup_{0\not =v\in H^1(\Omega)}\left(\frac{\product{\dt e_h}{v-P_h v}_{\Omega}}{\norm{v}{H^1(\Omega)}}  
+ \frac{\product{\dt e_h}{P_h v}_{\Omega}}{\norm{v}{H^1(\Omega)}} \right) \\
&\leq \sup_{0\not =v\in H^1(\Omega)}\frac{\norm{z_h^a}{H^1(\Omega)}
\norm{P_h v}{H^1(\Omega)}}{\norm{v}{H^1(\Omega)}} \leq C_P\norm{z_h^a}{H^1(\Omega)}
\end{align*}
with the constant $C_P>0$.
Thus $\norm{z_h^a}{H_T}$ is an upper bound for $\norm{\dt e_h}{H'_T}$.
The norm $\norm{\phi(t)-\phi_{h,\tau}(t)}{B}$ is also not computable.
Hence we may use the equivalent norm 
\begin{align*}
 \norm{\phi(t)-\phi_{h,\tau}(t)}{B} \sim  
 \norm{\phi(t)-\phi_{h,\tau}(t)}{\V} := \dual{\V(\phi(t)-\phi_{h,\tau}(t))}{\phi(t)-\phi_{h,\tau}(t)}_\Gamma,
\end{align*}
see~\cite{Erath:2010-phd} for details.
Thus $\norm{\phi-\phi_{h,\tau}}{L^2(0,T;\V)}$ is an equivalent norm to
$\norm{\phi-\phi_{h,\tau}}{B_T}$.
We approximate all other spatial norms by Gaussian quadrature or with the matrices from the discretization.
The time integral in the Bochner-Sobolev norms is also computed with a Gaussian quadrature.
For the energy norm we therefore present the upper bound 
\begin{align*}
 \norm{u-u_{h,\tau}}{Q_T}+\norm{\phi-\phi_{h,\tau}}{B_T}
 &\leq  (\norm{u-u_{h,\tau}}{H_T}^2+\norm{z_h^a}{H_T}^2)^{1/2}
 +\norm{\phi - \phi_{h,\tau}}{L^2(0,T;\V)}.
\end{align*}
Furthermore, with respect to the error splitting~\cref{eq:l2splitting1}--\cref{eq:l2splitting2}
we also calculate the error
\begin{align*}
(\norm{\overline{u}_h-u_{h,\tau}}{H_T}^2+\norm{z_h^a}{H_T}^2)^{1/2}
+\norm{\overline{\phi}_h - \phi_{h,\tau}}{L^2(0,T;\V)} 
\end{align*}
with the $L^2$-projected analytical solutions $\overline{u}_h(t)\in H^h$ of $u(t)$
and $\overline{\phi}_h(t)\in B^h$ of $\phi(t)$.

\subsubsection{Smooth solution}
\label{subsubsec:bsp1}
For the first example we use the interior solution
\begin{align*}
u(x_1,x_2,t)&=\sin(2\pi t)(1-100 x_1^2-100 x_2^2)e^{-50(x_1^2+x_2^2)}.
\end{align*}
Hence, both, $u$ and $u_e$ are smooth and according to
\cref{th:errorestimateorder}
we expect the optimal
convergence rate $\O(h+\tau)$ 
which is indeed observed in \cref{fig:errorbsp1}.
%
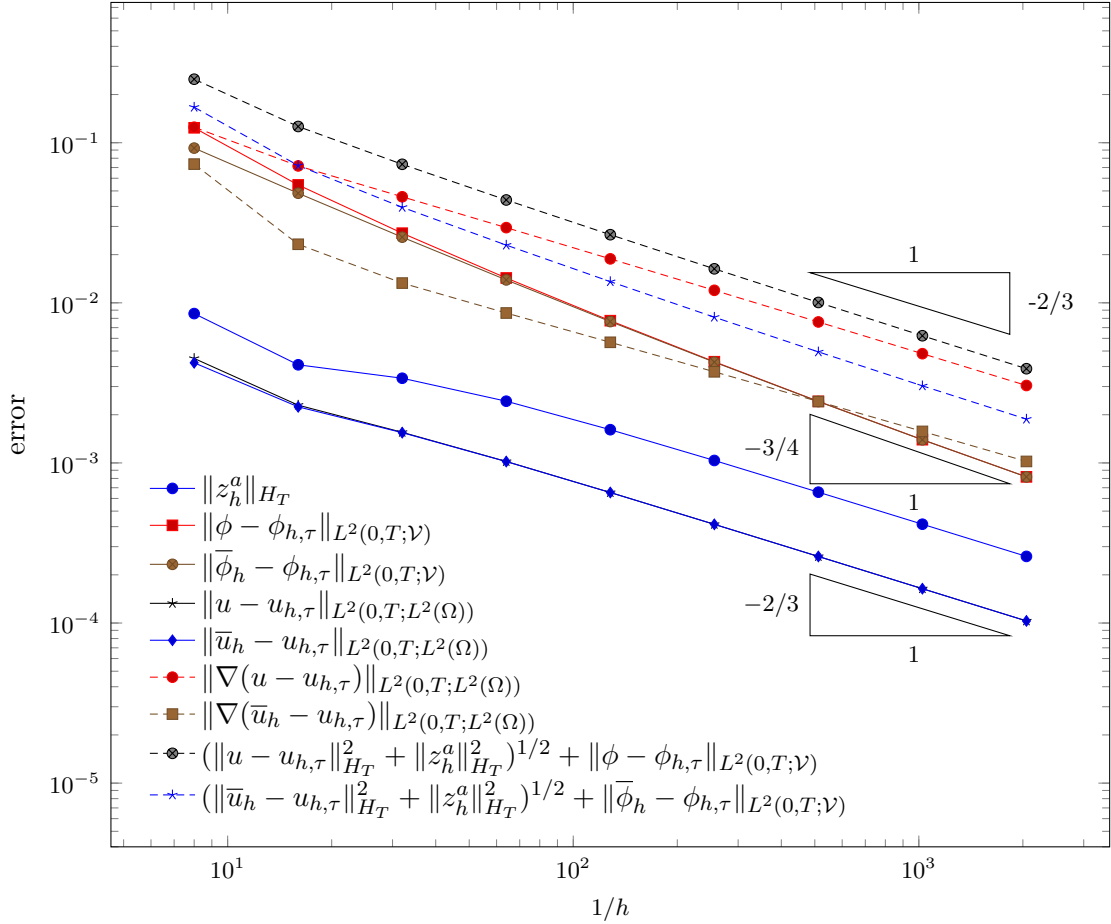
\begin{figure}
\centering
\begin{tikzpicture}
\begin{loglogaxis}[width=0.9\textwidth,
xlabel={$1/h$}, ylabel={\small error}, font={\scriptsize}, 
ymin=4e-6,
legend style={font=\small, draw=none, fill=none, cells={anchor=west}, legend pos=south west}]
\addplot table [x=invmaxMeshsizeh,y=errorH1dual] {figures/LShapeunsmooth19112017.dat};
\addplot table [x=invmaxMeshsizeh,y=errorenergyV] {figures/LShapeunsmooth19112017.dat};
\addplot table [x=invmaxMeshsizeh,y=errorenergyVproj] {figures/LShapeunsmooth19112017.dat};
\addplot table [x=invmaxMeshsizeh,y=errorL2] {figures/LShapeunsmooth19112017.dat};
\addplot table [x=invmaxMeshsizeh,y=errorL2proj] {figures/LShapeunsmooth19112017.dat};
\addplot table [x=invmaxMeshsizeh,y=errorH1semi] {figures/LShapeunsmooth19112017.dat};
\addplot table [x=invmaxMeshsizeh,y=errorH1semiproj] {figures/LShapeunsmooth19112017.dat};
\addplot table [x=invmaxMeshsizeh,y=globalEnergy] {figures/LShapeunsmooth19112017.dat};
\addplot table [x=invmaxMeshsizeh,y=globalEnergyproj] {figures/LShapeunsmooth19112017.dat};

\logLogSlopeTriangle{0.9}{0.2}{0.68}{-2/3}{black}{\scriptsize};
\logLogSlopeTrianglelow{0.9}{0.2}{0.25}{-2/3}{black}{\scriptsize};
\logLogSlopeTrianglelow{0.9}{0.2}{0.43}{-3/4}{black}{\scriptsize};

\legend{
$\norm{z_h^a}{H_T}$,
$\norm{\phi - \phi_{h,\tau}}{L^2(0,T;\V)}$,
$\norm{\overline{\phi}_h - \phi_{h,\tau}}{L^2(0,T;\V)}$,
$\norm{u-u_{h,\tau}}{L^2(0,T;L^2(\Omega))}$,
$\norm{\overline{u}_h-u_{h,\tau}}{L^2(0,T;L^2(\Omega))}$,
$\norm{\nabla(u-u_{h,\tau})}{L^2(0,T;L^2(\Omega))}$,
$\norm{\nabla(\overline{u}_h-u_{h,\tau})}{L^2(0,T;L^2(\Omega))}$,
$(\norm{u-u_{h,\tau}}{H_T}^2+\norm{z_h^a}{H_T}^2)^{1/2}+\norm{\phi - \phi_{h,\tau}}{L^2(0,T;\V)}$,
$(\norm{\overline{u}_h-u_{h,\tau}}{H_T}^2+\norm{z_h^a}{H_T}^2)^{1/2}+\norm{\overline{\phi}_h - \phi_{h,\tau}}{L^2(0,T;\V)}$}
\end{loglogaxis}
\end{tikzpicture}
\caption{The different error components of the solutions   
$u_{h,\tau}$ and $\phi_{h,\tau}$
for uniform refinement in time and space 
for the example with a spatial generic singularity of the interior solution in \cref{subsubsec:bsp2}.
The added energy error norms 
$(\norm{u-u_{h,\tau}}{H_T}^2+\norm{z_h^a}{H_T}^2)^{1/2}+\norm{\phi - \phi_{h,\tau}}{L^2(0,T;\V)}$
and
$(\norm{\overline{u}_h-u_{h,\tau}}{H_T}^2+\norm{z_h^a}{H_T}^2)^{1/2}+\norm{\overline{\phi}_h - \phi_{h,\tau}}{L^2(0,T;\V)}$
show the reduced convergence order as predicted in \cref{th:errorestimateorder}.
}
\label{fig:errorbsp2}
\end{figure}
\subsubsection{Generic singularity at the reentrant corner}
\label{subsubsec:bsp2}
For the second example, we choose the analytical solution 
\begin{align*}
u(x_1,x_2,t) &=(1+t^2) r^{2/3} \sin(2\varphi/3)
\end{align*}
with the polar coordinates
$(x_1,x_2) = r(\cos\varphi, \sin\varphi)$, $r\in\RR_+$ and $\varphi\in [0,2\pi)$.
This solution is a classical test solution in the spatial components and exhibits 
a generic singularity at the reentrant corner $(0,0)$ of $\Omega$. Note that $\Delta u = 0$
and that the function $u(x_1,x_2,\cdot)$ is only 
in $H^{1+2/3-\varepsilon}(\Omega)$ for $\varepsilon>0$. 
As analyzed 
in \cref{th:errorestimateorder} and observed
in \cref{fig:errorbsp2} 
we obtain a reduced convergence rate of $\O(h^{2/3}+\tau)$.
%

\begin{figure}
\centering
\begin{tikzpicture}
\begin{loglogaxis}[width=0.9\textwidth,
xlabel={\small $1/\tau$=\# time intervals}, ylabel={\small error}, font={\scriptsize}, 
ymin=1e-6,
legend style={font=\small, draw=none, fill=none, cells={anchor=west}, legend pos=south west}]
\addplot table [x=numberTimeintervals,y=errorH1dual] {figures/mexicanhatfem19112017.dat};
\addplot table [x=numberTimeintervals,y=errorenergyV] {figures/mexicanhatfem19112017.dat};
\addplot table [x=numberTimeintervals,y=errorenergyVproj] {figures/mexicanhatfem19112017.dat};
\addplot table [x=numberTimeintervals,y=errorL2] {figures/mexicanhatfem19112017.dat};
\addplot table [x=numberTimeintervals,y=errorL2proj] {figures/mexicanhatfem19112017.dat};
\addplot table [x=numberTimeintervals,y=errorH1semi] {figures/mexicanhatfem19112017.dat};
\addplot table [x=numberTimeintervals,y=errorH1semiproj] {figures/mexicanhatfem19112017.dat};
\addplot table [x=numberTimeintervals,y=globalEnergy] {figures/mexicanhatfem19112017.dat};
\addplot table [x=numberTimeintervals,y=globalEnergyproj] {figures/mexicanhatfem19112017.dat};

\logLogSlopeTriangle{0.925}{0.2}{0.49}{-1/3}{black}{\scriptsize};
\logLogSlopeTriangle{0.925}{0.2}{0.67}{-1}{black}{\scriptsize};
\logLogSlopeTriangle{0.925}{0.15}{0.35}{-3/2}{black}{\scriptsize};
\logLogSlopeTriangle{0.925}{0.15}{0.175}{-3/4}{black}{\scriptsize};

\legend{
$\norm{z_h^a}{H_T}$,
$\norm{\phi - \phi_{h,\tau}}{L^2(0,T;\V)}$,
$\norm{\overline{\phi}_h - \phi_{h,\tau}}{L^2(0,T;\V)}$,
$\norm{u-u_{h,\tau}}{L^2(0,T;L^2(\Omega))}$,
$\norm{\overline{u}_h-u_{h,\tau}}{L^2(0,T;L^2(\Omega))}$,
$\norm{\nabla(u-u_{h,\tau})}{L^2(0,T;L^2(\Omega))}$,
$\norm{\nabla(\overline{u}_h-u_{h,\tau})}{L^2(0,T;L^2(\Omega))}$,
$(\norm{u-u_{h,\tau}}{H_T}^2+\norm{z_h^a}{H_T}^2)^{1/2}+\norm{\phi - \phi_{h,\tau}}{L^2(0,T;\V)}$,
$(\norm{\overline{u}_h-u_{h,\tau}}{H_T}^2+\norm{z_h^a}{H_T}^2)^{1/2}+\norm{\overline{\phi}_h - \phi_{h,\tau}}{L^2(0,T;\V)}$}
\end{loglogaxis}
\end{tikzpicture}
\caption{The different error components of the solutions   
$u_{h,\tau}$ and $\phi_{h,\tau}$
for uniform refinement in time and space 
for the example with a singularity in the time component 
of the interior solution in \cref{subsubsec:bsp3}.
The added energy error norms 
$(\norm{u-u_{h,\tau}}{H_T}^2+\norm{z_h^a}{H_T}^2)^{1/2}+\norm{\phi - \phi_{h,\tau}}{L^2(0,T;\V)}$
and
$(\norm{\overline{u}_h-u_{h,\tau}}{H_T}^2+\norm{z_h^a}{H_T}^2)^{1/2}+\norm{\overline{\phi}_h - \phi_{h,\tau}}{L^2(0,T;\V)}$
show the reduced convergence order as predicted in \cref{th:errorestimateorder}.
}
\label{fig:errorbsp3}
\end{figure}
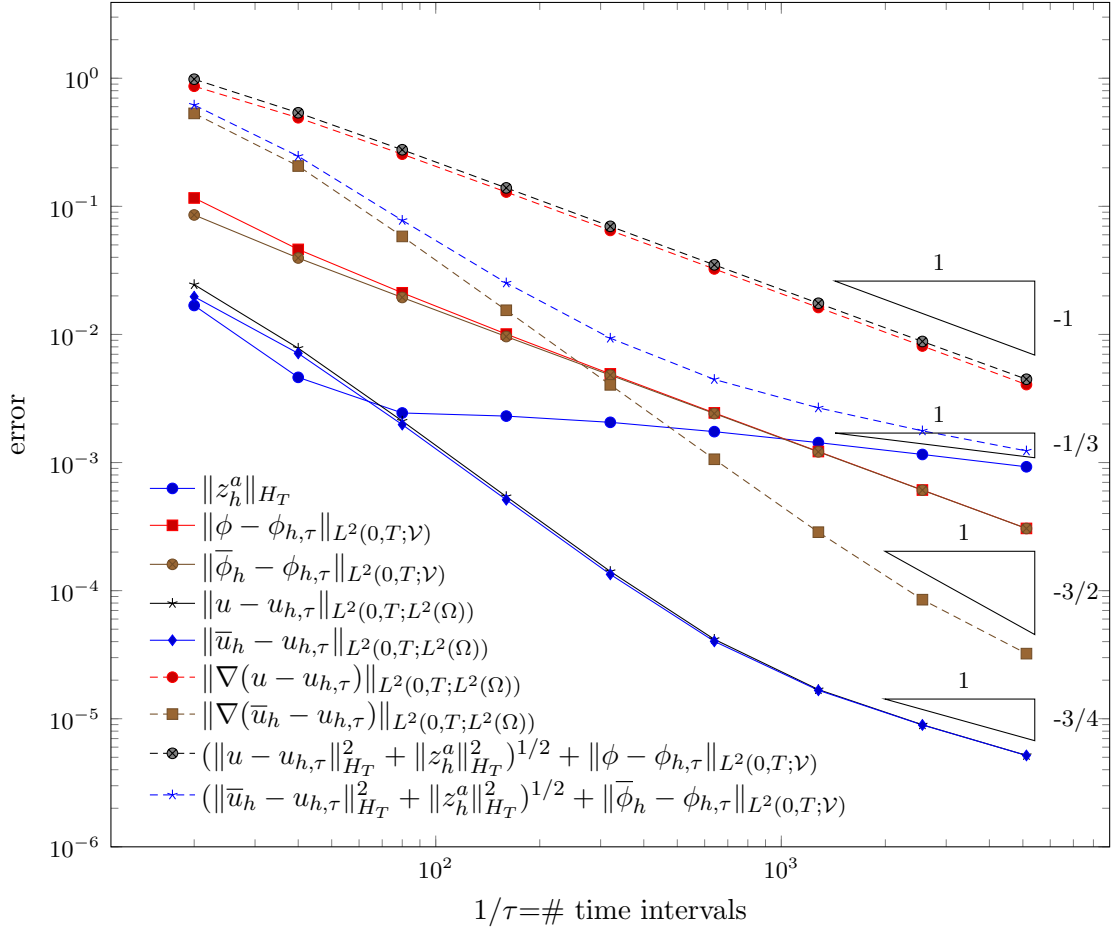

\subsubsection{Non-smooth function in time}
\label{subsubsec:bsp3}
The third example is less regular in time, but smooth in space, and reads
\begin{align*}
u(x_1,x_2,t) =t^{5/6} (1-100 x_1^2-100 x_2^2)e^{-50(x_1^2+x_2^2)}.
\end{align*}
Note that the function $u(x,\cdot)$ is only in $H^{4/3}(0,T)$. 
According to our analysis we expect
a convergence rate of $\O(h+\tau^{1/3})$. 
We plot the convergence order with respect to the number of time intervals ($=1/\tau$)
in \cref{fig:errorbsp3}.
Note that the energy norm error $\norm{u-u_{h,\tau}}{Q_T}+\norm{\phi-\phi_{h,\tau})}{B_T}$ represented by
$(\norm{u-u_{h,\tau}}{H_T}^2+\norm{z_h^a}{H_T}^2)^{1/2}+\norm{\phi - \phi_{h,\tau}}{L^2(0,T;\V)}$
seems to have a misleading convergence order of $\O(\tau)$. 
The error component $\norm{z_h^a}{H_T}$, representing
the dual norm error $\norm{\dt(u-u_{h,\tau})}{H'_T}$, has convergence order
$\O(\tau^{1/3})$. With respect to 
$\norm{\nabla (u-u_{h,\tau})}{L^2(0,T;L^2(\Omega))}$ this error component is rather small.
Hence the predicted convergence rate $\O(h+\tau^{1/3})$ would be observed asymptotically which can not
be visualized here due to computational restrictions. 
\begin{figure}[tbhp]
\centering 
\subfigure[Solution at $t=0.0125$.]{\label{fig:acap}\includegraphics[width=.3\textwidth]{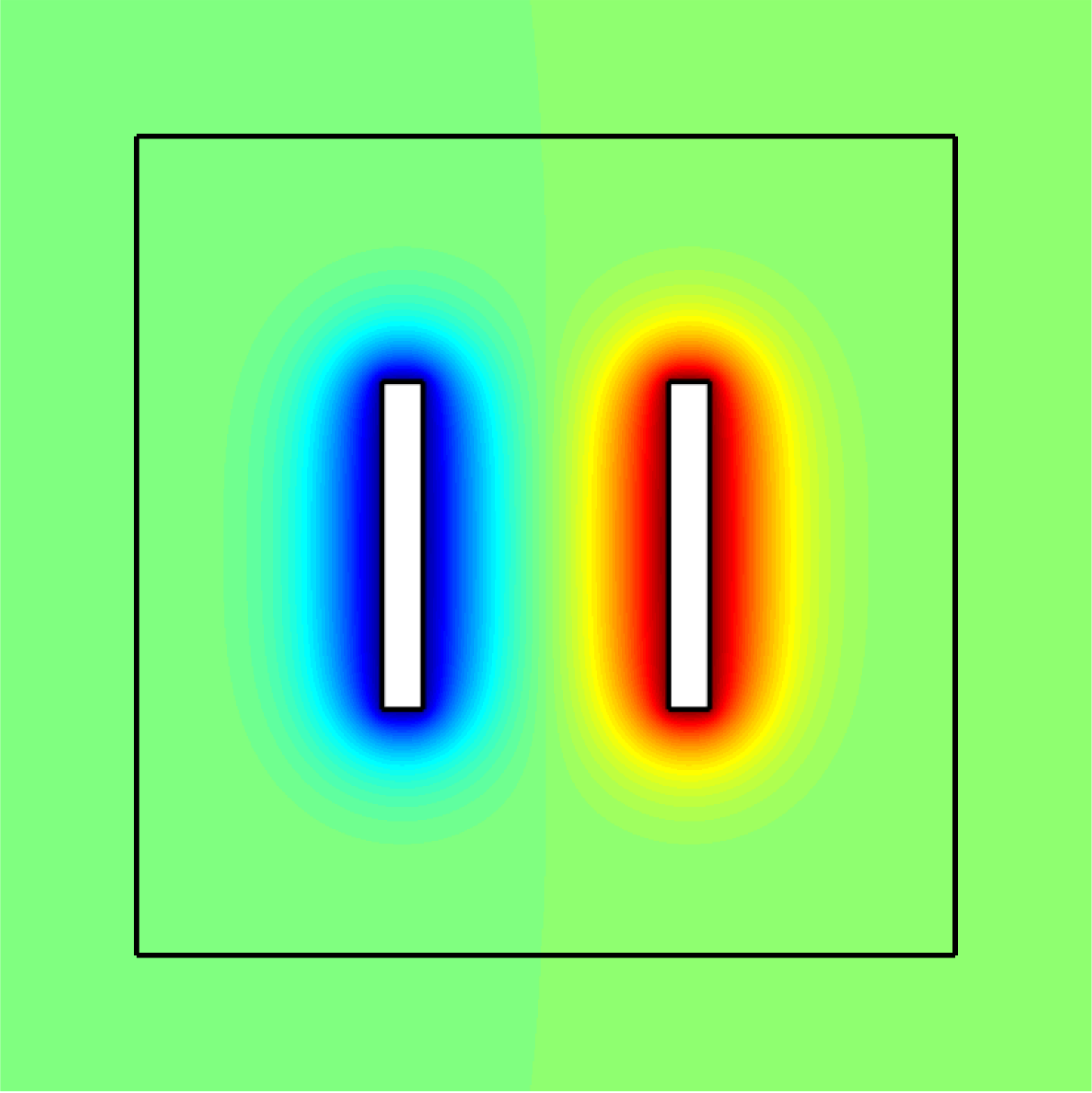}}
\hspace{.025\textwidth} 
\subfigure[Solution at $t=0.05$.]{\label{fig:bcap}\includegraphics[width=.3\textwidth]{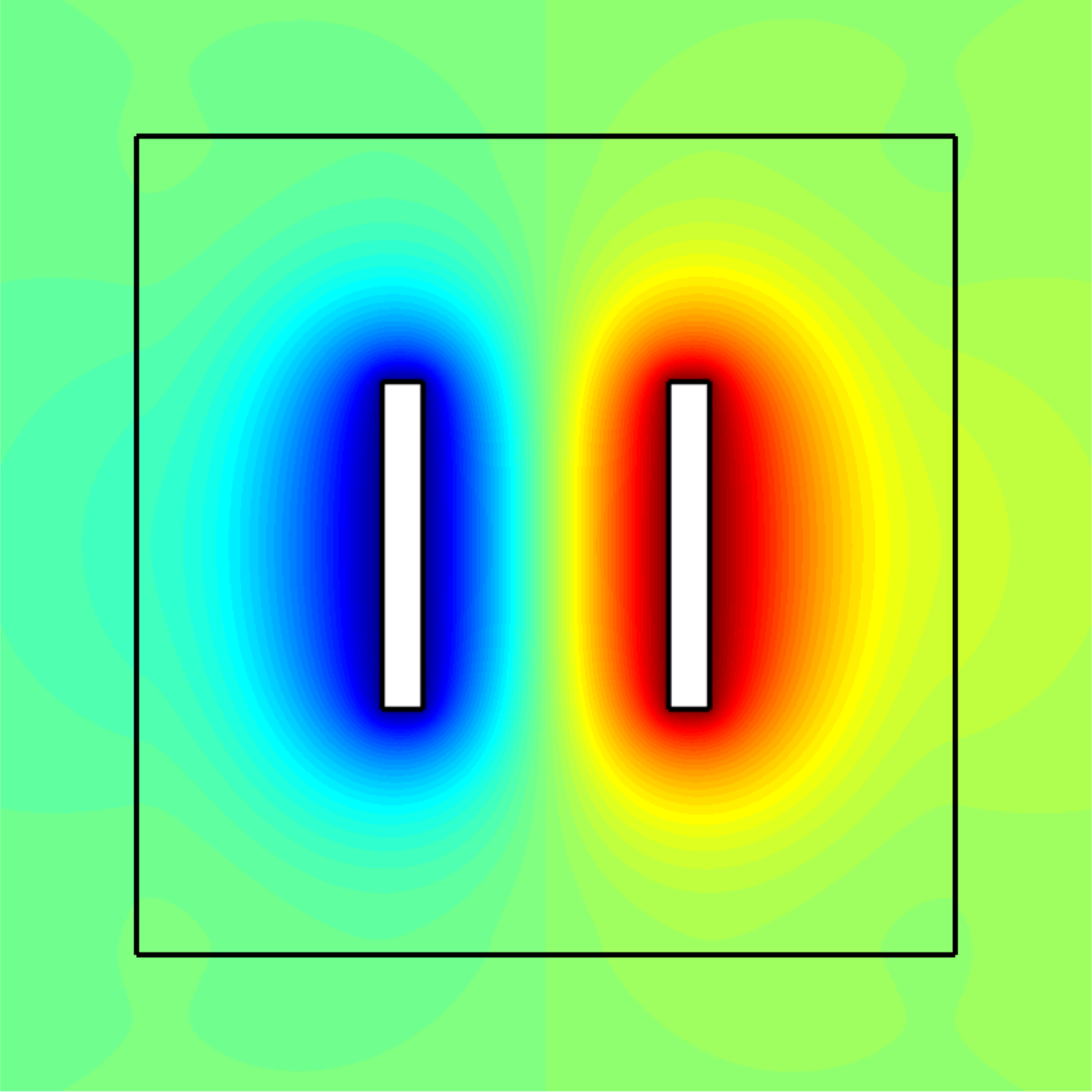}}
\hspace{.025\textwidth}  
\subfigure[Solution at $t=0.4875$.]{\label{fig:ccap}\includegraphics[width=.3\textwidth]{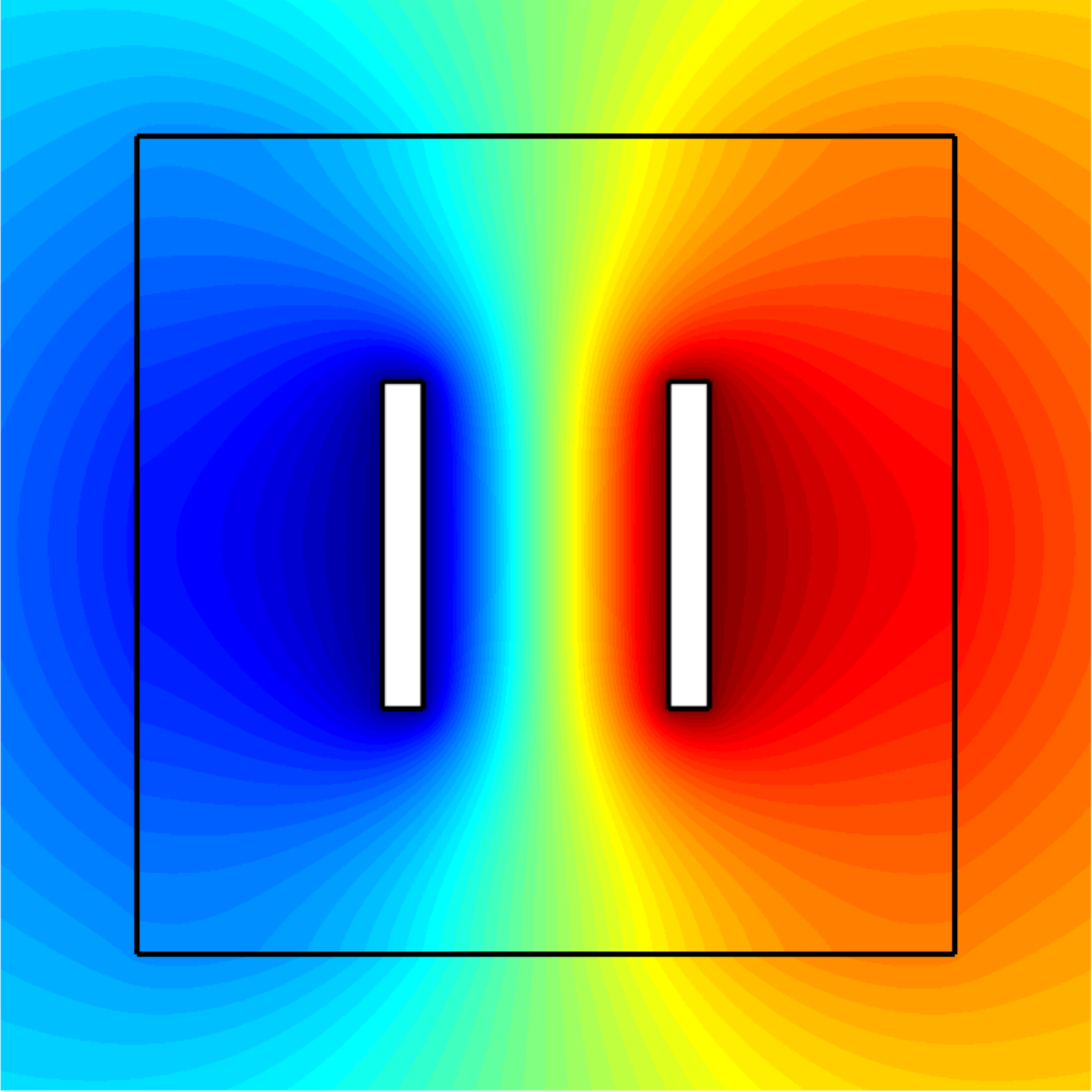}} 
\vspace{.1\baselineskip}
\subfigure[Solution at $t=0.5$.]{\label{fig:dcap}\includegraphics[width=.3\textwidth]{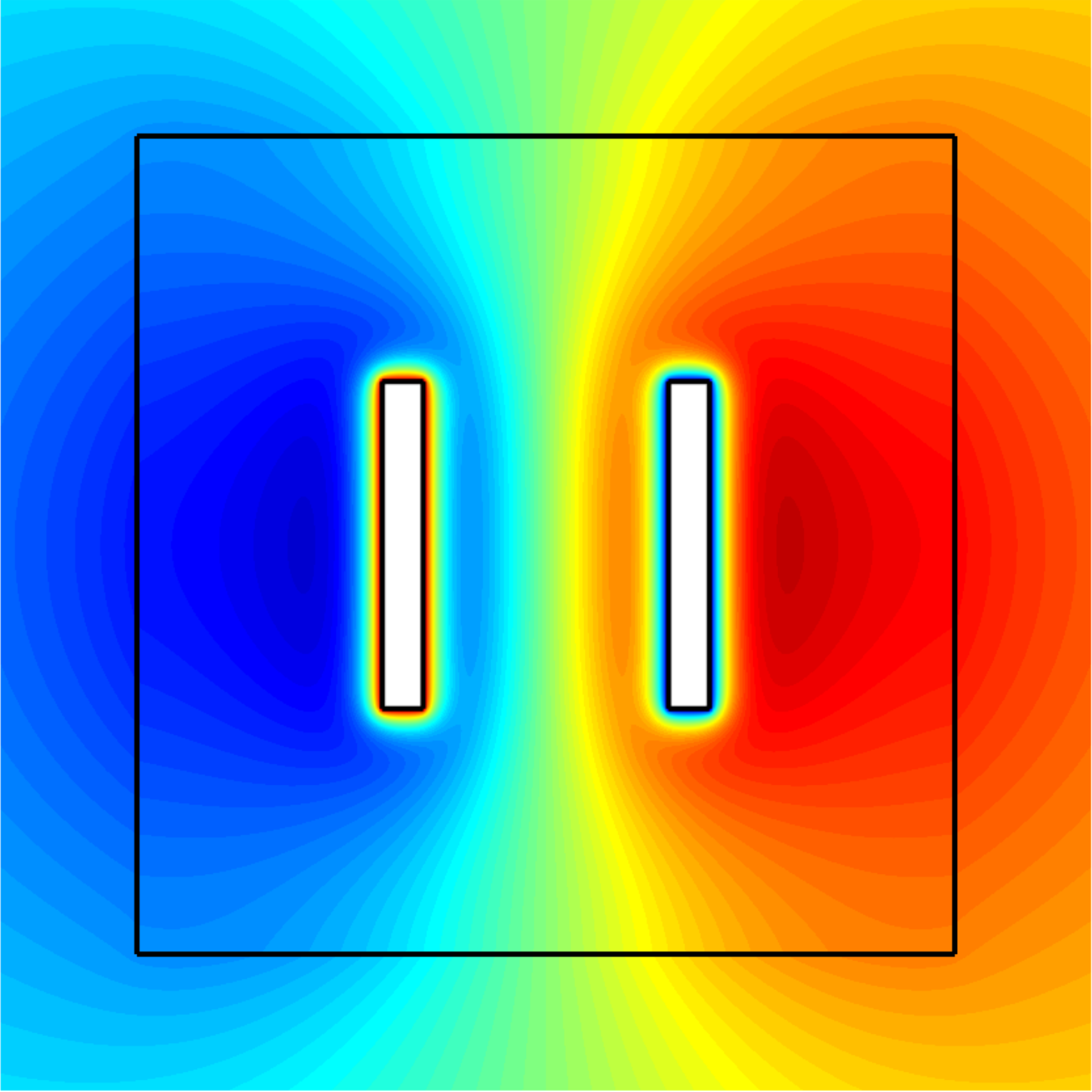}} 
\hspace{.025\textwidth} 
\subfigure[Solution at $t=0.6$.]{\label{fig:fcap}\includegraphics[width=.3\textwidth]{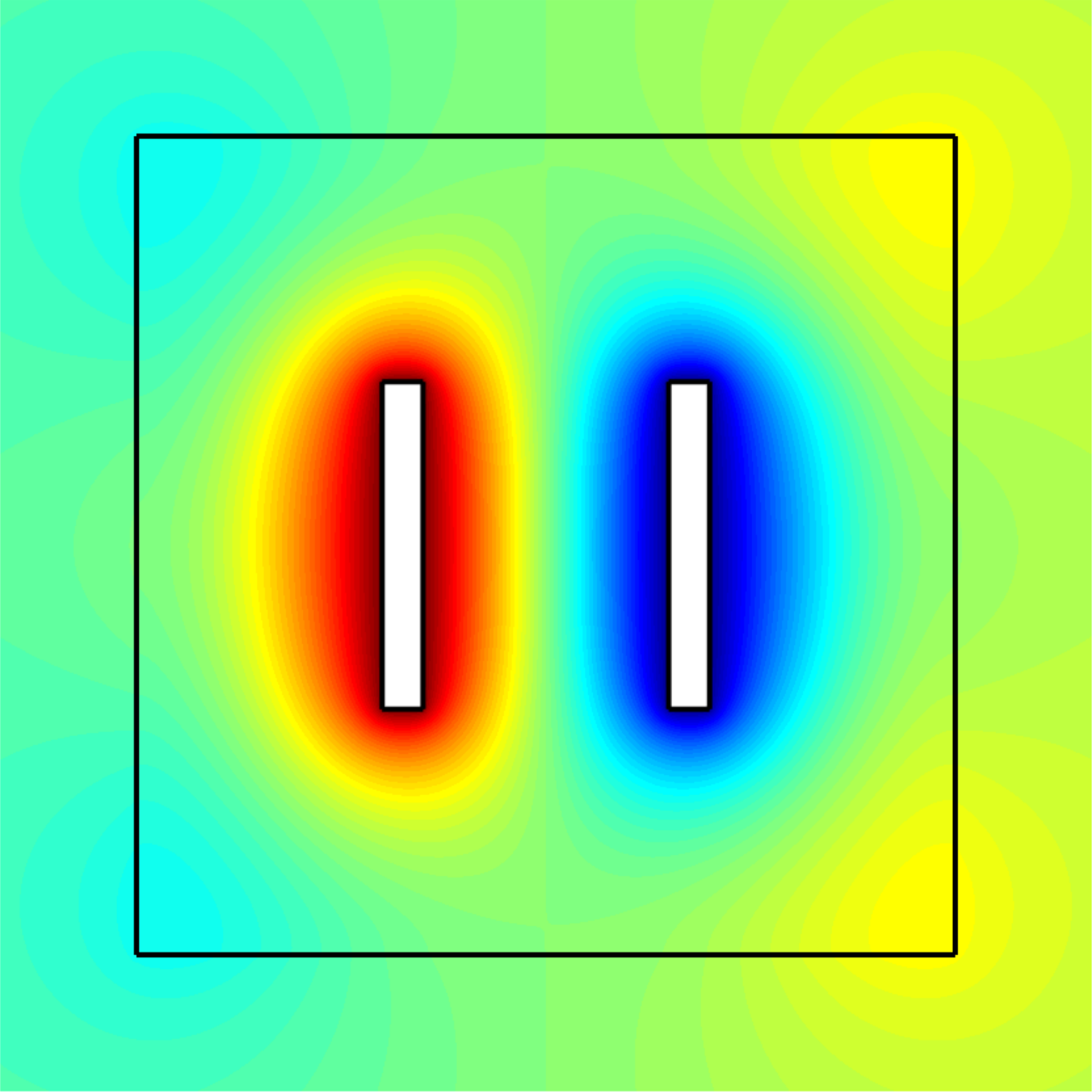}}
\hspace{.025\textwidth}  
\subfigure[Solution at $t=1.0$.]{\label{fig:gcap}\includegraphics[width=.3\textwidth]{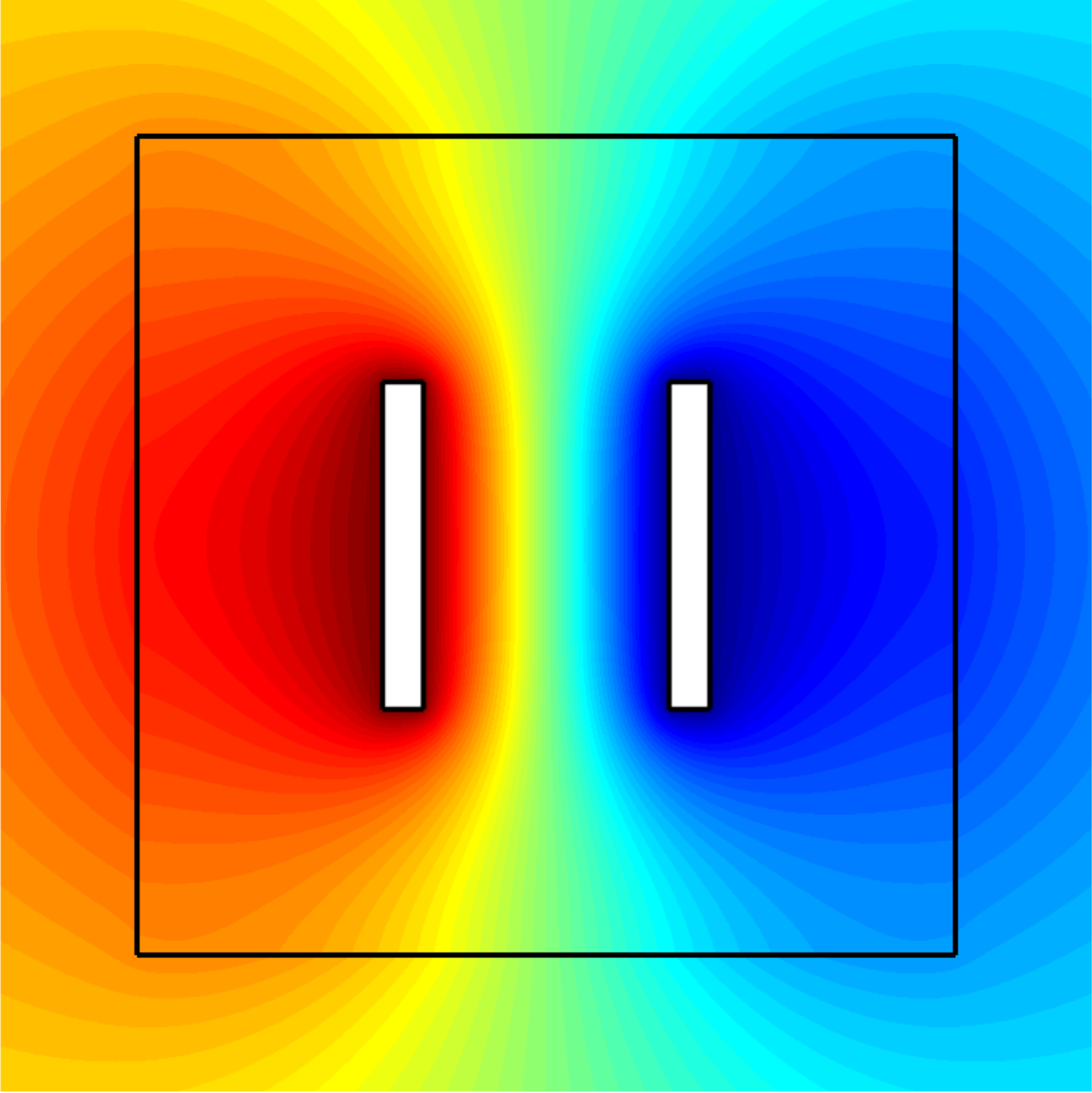}}
\vspace{.25\baselineskip}
\includegraphics[width=0.5\textwidth,clip]{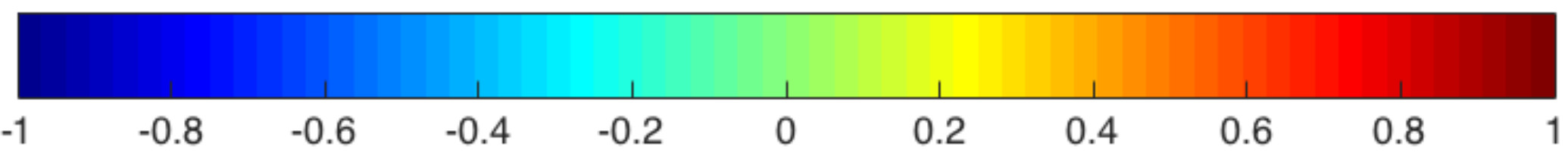}
\caption{Solution of the capacitor example in \cref{subsec:bspcap} at different times.}
\label{fig:cap}
\end{figure} 

\begin{figure}[tbhp]
\centering 
\includegraphics[width=.7\textwidth]{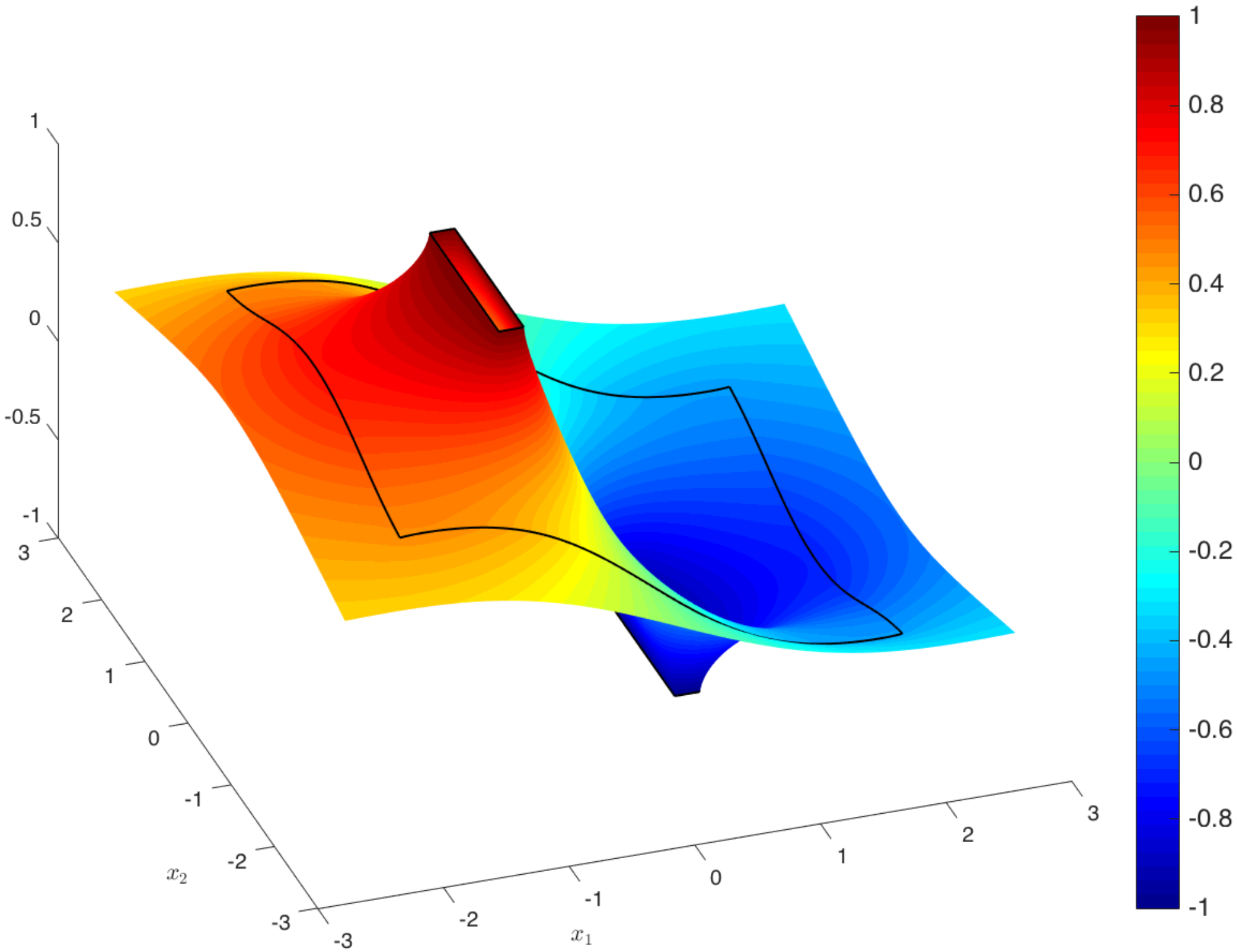}
\caption{Solution of the capacitor example in \cref{subsec:bspcap} at the end $T=1$.}
\label{fig:cap3D}
\end{figure} 
 
 \subsection{Quasi-electrostatic problem}\label{subsec:bspcap}
In the last example we want to apply our numerical scheme to a more practical
problem~\cite[Example 8.2]{Carstensen:1999}. 
The idea behind the problem 
is to model the potential of a capacitor
in an unbounded domain
with two electrodes $\Omega_{D,1}=[-0.8,-0.6]\times [-0.8,0,8]$
and $\Omega_{D,2}=[0.6,0.8]\times [-0.8,0,8]$.
For this we consider our model problem~\cref{eq:model1}--\cref{eq:model6}.
with the interior domain $\Omega=(-2,2)^2\backslash\big(\Omega_{D,1}\cup \Omega_{D,2}\big)$
and the exterior domain $\Omega_e=\mathbb{R}^2\backslash [-2,2]^2$, see also \cref{subfig:meshcap}.
We choose $\tilde f=0$, $\tilde g=0$, $\tilde h=0$, and the initial field $u(\cdot,0)=0$.
Contrary to~\cref{eq:model1} we allow a diffusion coefficient in the interior domain $\Omega$ 
of $5$ instead of $1$. 
Furthermore, we define $\Gamma_{D,1}:=\partial\Omega_{D,1}$ and $\Gamma_{D,2}:=\partial\Omega_{D,2}$.
Thus the coupling boundary reads 
$\Gamma=\partial \Omega_e=\partial \Omega\backslash\big(\Gamma_{D,1}\cup\Gamma_{D,2}\big)$.
For the charge at the electrode boundaries $\Gamma_{D,1}$ and $\Gamma_{D,2}$,
which are Dirichlet boundaries in the model problem, we
choose
\begin{alignat}{2}
u(x,t)&= 
\begin{cases} -1 &\text{for}\quad t<0.5 \\ 
\phantom{-}1 &\text{for}\quad t \geq 0.5 
\end{cases} 
&\qquad& \text{on } \Gamma_{D,1} \times (0,1), \\
 u(x,t)&= 
 \begin{cases} \phantom{-}1 &\text{for}\quad t<0.5 \\ 
 -1 &\text{for}\quad t \geq 0.5 
 \end{cases}     
 &\qquad& \text{on } \Gamma_{D,2} \times (0,1).
\end{alignat}
Hence the charges are fixed to $\pm 1$ 
at Dirichlet boundary $\Gamma_{D,1} \cup \Gamma_{D,2}$
and the polarity is reversed at $t=0.5$.
In \cref{fig:cap} we plot the interior and part of the exterior solution at different times after
$5$ uniform refinements of the triangulation \cref{subfig:meshcap}, i.e., $h=0.03125$
and $\tau=0.0015625$. We use the representation formula~\cref{eq:repformular} with the discrete solution
$u_{h,\tau}|_\Gamma$ and $\phi_{h,\tau}$ to get 
the approximation of $u_e$ in $\Omega_e$.  
The figure sequence shows how the electrical field is building up and evolves after the change of polarity.
Finally, we plot the solution at the end time $T=1$ in \cref{fig:cap3D}.
 
\section{Conclusions}

In this work we provided a refined a~priori analysis for the 
semi-discretization of the non-symmetric FEM-BEM coupling
for a parabolic-elliptic interface problem.
Furthermore, the first a~priori analysis was worked out 
for the full discretization 
of this coupling type 
in terms of the energy norm
of the solution space.
We were able to show quasi-optimality results for both, 
the semi- and the full discretization, 
with a Galerkin method in 
space and a variant of the implicit Euler method in time. 
Then we utilized the piecewise linear ansatz function space 
and the piecewise constant ansatz function space to approximate
the interior problem and the exterior problem, respectively. 
This defines a classical non-symmetric FEM-BEM coupling approach with
first order convergence.
Note that this is the optimal convergence rate for these ansatz spaces in this norm. 
However, the optimal convergence rate in the $L^2$ norm,
which usually relies on a duality argument, still remains open. 
In case of a non-symmetric approach, adjoint regularity cannot be obtained as 
easy as in the symmetric case. 
Thus our analysis avoided using the elliptic projection and 
used the $L^2$-projection instead.
Numerical experiments confirmed the theoretical findings. In particular they show
that our method even converges on non-convex domains with less regular data. 

\bibliographystyle{alpha}
\bibliography{literature}  

\newcommand{\etalchar}[1]{$^{#1}$}
\begin{thebibliography}{AEF{\etalchar{+}}14}

\bibitem[AEF{\etalchar{+}}14]{HILBERT:2013-1}
M.~Aurada, M.~Ebner, M.~Feischl, S.~Ferraz-Leite, T.~F{\"u}hrer, P.~Goldenits,
  M.~Karkulik, M.~Mayr, and D.~Praetorius.
\newblock {HILBERT} --- a {MATLAB} implementation of adaptive {2D-BEM}.
\newblock {\em Numer. Algor.}, 67:1--32, 2014.

\bibitem[AFF{\etalchar{+}}13]{Aurada:2013-1}
M.~Aurada, M.~Feischl, T.~F{\"u}hrer, M.~Karkulik, J.~M. Melenk, and
  D.~Praetorius.
\newblock Classical {FEM-BEM} coupling methods: nonlinearities, well-posedness,
  and adaptivity.
\newblock {\em Comput. Mech.}, 51(4):399--419, 2013.

\bibitem[BPS01]{Bramble:2001}
J.~H. Bramble, J.~E. Pasciak, and O.~Steinbach.
\newblock On the stability of the {L}2-projection in {H}1.
\newblock {\em Mathematics of Computation}, 71:147--156, 2001.

\bibitem[BS08]{Brenner:2008-book}
S.~C. Brenner and L.~R. Scott.
\newblock {\em The Mathematical Theory of Finite Element Methods}.
\newblock Springer, 3rd edition, 2008.

\bibitem[BY14]{Bank:2014}
R.~E. Bank and H.~Yserentant.
\newblock On the {$ H^1 $}-stability of the {$L^2$}-projection onto finite
  element spaces.
\newblock {\em Numer. Math.}, 126:361--381, 2014.

\bibitem[CES90]{Costabel:1990}
M.~Costabel, V.~J. Ervin, and E.~P. Stephan.
\newblock Symmetric coupling of finite elements and boundary elements for a
  parabolic-elliptic interface problem.
\newblock {\em Quarterly of Applied Mathematics}, 48:265--279, 1990.

\bibitem[CF99]{Carstensen:1999}
C.~Carstensen and S.~A. Funken.
\newblock Coupling of nonconforming finite elements and boundary elements {I}:
  A priori estimates.
\newblock {\em Computing}, 62:229--241, 1999.

\bibitem[Cia78]{Ciarlet:1978-book}
P.~G. Ciarlet.
\newblock {\em The finite element method for elliptic problems}.
\newblock Studies in mathematics and its applications. North-Holland,
  Amsterdam, New-York, 1978.

\bibitem[Cos88a]{Costabel:1988-1}
M.~Costabel.
\newblock Boundary integral operators on {L}ipschitz domains: elementary
  results.
\newblock {\em SIAM J. Math. Anal.}, 19:613--626, 1988.

\bibitem[Cos88b]{Costabel:1988-2}
M.~Costabel.
\newblock A symmetric method for the coupling of finite elements and boundary
  elements.
\newblock In {\em The mathematics of finite elements and applications, {VI}
  ({U}xbridge, 1987)}, pages 281--288. Academic Press, London, 1988.

\bibitem[DL92]{Dautray:1992-5}
R.~Dautray and J.-L. Lions.
\newblock {\em Mathematical analysis and numerical methods for science and
  technology. {V}ol. 5}.
\newblock Springer-Verlag, Berlin, 1992.

\bibitem[EOS17]{Erath:2017-1}
C.~Erath, G.~Of, and F.-J. Sayas.
\newblock A non-symmetric coupling of the finite volume method and the boundary
  element method.
\newblock {\em Numer. Math.}, 135:895--922, 2017.

\bibitem[Era10]{Erath:2010-phd}
C.~Erath.
\newblock {\em {Coupling of the Finite Volume Method and the Boundary Element
  Method - Theory, Analysis, and Numerics}}.
\newblock PhD thesis, University of Ulm, 2010.

\bibitem[ES17]{ErathSchorr:2017-2}
C.~Erath and R~Schorr.
\newblock A simple boundary approximation for the non-symmetric coupling of
  finite element method and boundary element method for parabolic-elliptic
  interface problems.
\newblock {\em Preprint, to appear in ENUMATH}, 2017.

\bibitem[Eva10]{Evans:2010-book}
L.~C. Evans.
\newblock {\em Partial Differential Equations}.
\newblock Graduate studies in mathematics. American Mathematical Society, 2010.

\bibitem[Gon06]{Gonzalez:2006}
M.~Gonz\'{a}lez.
\newblock Fully discrete fem-bem method for a class of exterior nonlinear
  parabolic-elliptic problems in 2d.
\newblock {\em Appl. Numer. Math.}, 56(10):1340--1355, 2006.

\bibitem[JN80]{Johnson:1980-1}
C.~Johnson and J.~C. N{\'e}d{\'e}lec.
\newblock {On the coupling of boundary integral and finite element methods}.
\newblock {\em Math. Comput.}, 35:1063--1079, 1980.

\bibitem[McL00]{McLean:2000-book}
W.~McLean.
\newblock {\em Strongly Elliptic Systems and Boundary Inte gral Equations}.
\newblock Cambridge University Press, 2000.

\bibitem[MS87]{MacCamy:1987}
R.~C. MacCamy and M.~Suri.
\newblock A time-dependent interface problem for two-dimensional eddy currents.
\newblock {\em Quart. Appl. Math.}, 44:675--690, 1987.

\bibitem[Say09]{Sayas:2009-1}
F.-J. Sayas.
\newblock {The validity of Johnson-N\'ed\'elec's BEM-FEM coupling on polygonal
  interfaces}.
\newblock {\em SIAM J. Numer. Anal.}, 47:3451--3463, 2009.

\bibitem[Ste08]{Steinbach:2008-book}
O.~Steinbach.
\newblock {\em Numerical Approximation Methods for Elliptic Boundary Value
  Problems: Finite and Boundary Elements}.
\newblock Texts in applied mathematics. Springer, New York, 2008.

\bibitem[Ste11]{Steinbach:2011}
O.~Steinbach.
\newblock A note on the stable one-equation coupling of finite and boundary
  elements.
\newblock {\em SIAM J. Numer. Anal.}, 49:1521--1531, 2011.

\bibitem[Tan14]{Tantardini:2014-1}
F.~Tantardini.
\newblock {\em Quasi-Optimality in the Backward Euler-Galerkin Method for
  Linear Parabolic Problems}.
\newblock PhD thesis, Universit{\`a} degli Studi di Milano, Milan, Italy, 2014.

\bibitem[Var71]{Varga:1971-book}
R.~S. Varga.
\newblock {\em Functional Analysis and Approximation Theory in Numerical
  Analysis}.
\newblock CBMS-NSF Regional Conference Series in Applied Mathematics. SIAM,
  Philadelphia, 1971.

\bibitem[Whe73]{Wheeler:1973}
M.~F. Wheeler.
\newblock A priori {L}2 error estimates for {G}alerkin approximations to
  parabolic partial differential equations.
\newblock {\em SIAM J. Numer. Anal.}, 10:723--759, 1973.

\end{thebibliography}

\end{document}